\documentclass[12pt]{amsart}
\usepackage{euscript,epsf,amssymb,xr}

\let\oldmarginpar\marginpar
\renewcommand\marginpar[1]
{\oldmarginpar{\tiny\bf \begin{flushleft} #1 \end{flushleft}}}

\input xy
\xyoption{all}

\newcommand{\la}{\langle}
\newcommand{\ra}{\rangle}

\newtheorem{theorem}{\bf Theorem}[section]
\newtheorem{lemma}[theorem]{\bf Lemma}

\newtheorem{remark}[theorem]{\bf Remark}

\newtheorem{corollary}[theorem]{\bf Corollary}

\newcommand{\CC}{{\Bbb C}}

\newcommand{\CP}{{\Bbb CP}}

\newcommand{\FF}{{\Bbb F}}
\newcommand{\GG}{{\Bbb G}}

\newcommand{\QQ}{{\Bbb Q}}
\newcommand{\RR}{{\Bbb R}}

\newcommand{\ZZ}{{\Bbb Z}}

\newcommand{\ggreat}{>\kern-.7ex>}
\newcommand{\ssmall}{<\kern-.7ex<}
\newcommand{\qu}{/\kern-.7ex/}
\newcommand{\exh}{\to\kern-1.8ex\to}

\newcommand{\aA}{{\EuScript{A}}}
\newcommand{\bB}{{\EuScript{B}}}
\newcommand{\cC}{{\EuScript{C}}}
\newcommand{\dD}{{\EuScript{D}}}

\newcommand{\gG}{{\EuScript{G}}}
\newcommand{\hH}{{\EuScript{H}}}

\newcommand{\jJ}{{\EuScript{J}}}

\newcommand{\mM}{{\EuScript{M}}}

\newcommand{\sS}{{\EuScript{S}}}
\newcommand{\tT}{{\EuScript{T}}}

\newcommand{\xX}{{\EuScript{X}}}

\newcommand{\GL}{\operatorname{GL}}

\newcommand{\Aut}{\operatorname{Aut}}

\newcommand{\Diff}{\operatorname{Diff}}
\newcommand{\e}{\operatorname{e}}

\newcommand{\Ham}{\operatorname{Ham}}
\newcommand{\Hom}{\operatorname{Hom}}

\newcommand{\Id}{\operatorname{Id}}

\newcommand{\Jor}{\operatorname{Jor}}

\newcommand{\Ker}{\operatorname{Ker}}
\newcommand{\Map}{\operatorname{Map}}

\newcommand{\Perm}{\operatorname{Perm}}

\newcommand{\rk}{\operatorname{rk}}
\newcommand{\SO}{\operatorname{SO}}
\newcommand{\SU}{\operatorname{SU}}

\newcommand{\Symp}{\operatorname{Symp}}

\newcommand{\U}{\operatorname{U}}

\newcommand{\ov}{\overline}

\newcommand{\wh}{\widehat}

\newcommand{\ord}{\operatorname{ord}}

\newcommand{\imag}{{\mathbf i}}

\title[Finite subgroups of Ham and Symp]{Finite subgroups of Ham and Symp}

\author{Ignasi Mundet i Riera}
\address{Departament d'\`Algebra i Geometria\\
Facultat de Matem\`atiques\\
Universitat de Barcelona\\
Gran Via de les Corts Catalanes 585\\
08007 Barcelona \\
Spain}
\email{ignasi.mundet@ub.edu}

\date{June 13, 2017}

\thanks{This work has been partially supported by the (Spanish) MEC Project MTM2012-38122-C03-02.}

\begin{document}

\maketitle

\begin{abstract}
Let $(X,\omega)$ be a compact symplectic manifold of dimension
$2n$ and let $\Ham(X,\omega)$ be its group of Hamiltonian
diffeomorphisms. We prove the existence of a constant $C$,
depending on $X$ but not on $\omega$, such that any finite
subgroup $G\subset\Ham(X,\omega)$ has an abelian subgroup
$A\subseteq G$ satisfying $[G:A]\leq C$, and $A$ can be
generated by $n$ elements or fewer. If $b_1(X)=0$ we prove an
analogous statement for the entire group of symplectomorphisms
of $(X,\omega)$. If $b_1(X)\neq 0$ we prove the
existence of a constant $C'$ depending only on $X$ such that any
finite subgroup $G\subset\Symp(X,\omega)$ has a subgroup
$N\subseteq G$ which is either abelian or $2$-step nilpotent
and which satisfies $[G:N]\leq C'$.

These results are deduced from the classification of the finite simple groups, the topological rigidity of hamiltonian loops, and the following theorem, which we prove in this paper. Let $E$ be a complex vector bundle over a compact, connected, smooth and oriented manifold $M$; suppose that the real rank of $E$ is equal to the dimension of $M$, and that $\langle e(E),[M]\rangle\neq 0$, where $e(E)$ is the Euler class of $E$; then there exists a constant $C''$ such that, for any prime $p$ and any finite $p$-group $G$ acting on $E$ by vector bundle automorphisms preserving an almost complex structure on $M$, there is a subgroup $G_0\subseteq G$ satisfying $M^{G_0}\neq\emptyset$ and $[G:G_0]\leq C''$.
\end{abstract}

\tableofcontents

\section{Introduction}

\subsection{}
For any symplectic manifold $(X,\omega)$ denote by
$\Ham(X,\omega)$ (resp. $\Symp(X,\omega)$) the group of
Hamiltonian (resp. symplectic) diffeomorphisms of $(X,\omega)$.
In this paper we are interested on the finite subgroups of
$\Ham(X,\omega)$ or $\Symp(X,\omega)$ for arbitrary compact
$X$, that is, the Hamiltonian or symplectic finite
transformation groups of compact symplectic manifolds.

Let $\Symp_0(X,\omega)\subseteq\Symp(X,\omega)$
and $\Diff_0(X)\subseteq\Diff(X)$ denote the identity components.
Let
$$\hH(X,\omega),\quad \sS(X,\omega),\quad \dD(X)$$
denote the set of isomorphism classes of
finite subgroups of $\Ham(X,\omega)$, $\Symp_0(X,\omega)$, $\Diff_0(X)$ respectively.
There are obvious inclusions
$\hH(X,\omega)\subseteq \sS(X,\omega)\subseteq \dD(X)$ for any $X$.
A consequence of our results combined with those in \cite{M4} is
the existence of infinitely many compact symplectic manifolds
$(X,\omega)$ for which both $\hH(X,\omega)$ and $\dD(X)\setminus \hH(X,\omega)$ are infinite.
We also deduce the existence of infinitely many compact symplectic manifolds $(X,\omega)$
for which both $\hH(X,\omega)$ and $\sS(X,\omega)\setminus \hH(X,\omega)$ are infinite.

One can easily construct plenty of examples of Hamiltonian finite
transformation groups of compact manifolds by taking any compact symplectic manifold
$X$ endowed with a Hamiltonian action of a connected Lie group
$K$ and restricting the action to the finite subgroups of $K$.
In particular for any finite group $G$ there is some compact
symplectic manifold $X$ on which $G$ acts effectively by
Hamiltonian diffeomorphisms (embed $G$ in $\U(n)$ for some $n$
and take $X$ to be any Hamiltonian $\U(n)$-manifold, e.g.
$\CP^n$). Of course, this construction is far from capturing
all Hamiltonian finite transformation groups: one should not
expect an arbitrary finite subgroup of $\Ham(X,\omega)$ to be
included in a compact connected subgroup of $\Ham(X,\omega)$.
On the negative side, there are examples of compact symplectic
manifolds whose group of Hamiltonian diffeomorphisms is known
to contain no nontrivial finite subgroup; this is the case, for
example, if $\pi_2=0$, as proved by Polterovich \cite[Theorem
1.2]{Pol}.

Roughly speaking, our main result says that for any compact
symplectic manifold $X$ there is an upper bound (depending only
on the topology of $X$) on {\it how much nonabelian} the finite
subgroups of $\Ham(X,\omega)$ are. If $X$ is simply connected
we prove the same property for finite subgroups of
$\Symp(X,\omega)$. In contrast, there are many examples of
compact symplectic manifolds $(X,\omega)$ admitting no such
bound for finite subgroups of $\Diff(X)$, so the properties we
prove on $\Ham(X,\omega)$ are genuinely symplectic. Similarly,
if $X$ is a non simply connected symplectic manifold it is
known \cite{M5} that there is in general no upper bound
depending only on $X$ on {\it how much nonabelian} the finite
subgroups of $\Symp(X,\omega)$ are (although an upper bound
depending on $\omega$ might exist in all situations). Hence our
result on $\Ham(X,\omega)$ is sharp.

As a measure of how much nonabelian a finite group is, we take the index of
an abelian subgroup of maximal size. So the question we are interested in is
the so-called Jordan property for Hamiltonian and symplectomorphism groups, which we
now recall.

Let $C$ be a natural number. A group $\gG$ is $C$-Jordan if
each finite subgroup $G\subseteq\gG$ has an abelian subgroup
$A\subseteq G$ satisfying $[G:A]\leq C$. A group $\gG$ is
Jordan if it is $C$-Jordan for some $C$. This terminology,
which was introduced by V. Popov in \cite{Po0}, is inspired by
a classical theorem of C. Jordan \cite{J} which states that
$\GL(n,\CC)$ is Jordan for any $n$.

We will consider the following more refined notion. Let $C,d$
be natural numbers. A group $\gG$ is $(C,d)$-Jordan if each
finite subgroup $G\subseteq\gG$ has an abelian subgroup
$A\subseteq G$ satisfying $[G:A]\leq C$ and such that $A$ can
be generated by $d$ (or fewer) elements.

Some twenty years ago \'E. Ghys raised the question of whether
the diffeomorphism group of any smooth compact manifold is
Jordan. A number of papers have recently been devoted to Ghys's
question (see e.g. \cite{M3} and the references therein), and
it is known that the answer is affirmative in many cases and
negative in many other ones \cite{CPS,M4}, although a
description of which compact smooth manifolds have Jordan
diffeomorphism group seems to be at present a widely open
problem. As we will see, this is in sharp contrast with the
analogous question for Hamiltonian diffeomorphism groups of
compact symplectic manifolds, which are always Jordan.
Regarding our more refined notion of Jordanness, we remark that
a theorem of Mann and Su \cite{MS} (see Theorem
\ref{thm:MS} below) implies that if $X$ is a compact smooth
manifold, $\gG$ is any subgroup of $\Diff(X)$,  and $\gG$ is
$C$-Jordan for some $C$, then $\gG$ is also $(C,d)$-Jordan for
some $d$ depending only on $X$.

Analogous questions have been studied in algebraic geometry.
J.--P. Serre \cite{Se} proved that the classical Cremona group
is Jordan, and asked whether its higher dimensional analogues
are also Jordan. More generally, one may ask whether the
automorphism or birational transformation groups of general
affine or projective varieties are Jordan (see e.g. the survey
\cite{Po2} and, more recently, \cite{Bi,MZ,PS1,PS2}).

The question has also been studied previously in symplectic
geometry. The symplectomorphism group of $(T^2\times
S^2,\omega)$ has been proved to be Jordan for any choice of
symplectic structure $\omega$ in \cite{M5}. An interesting
consequence of this is that it gives an example of a symplectic
manifold with Jordan symplectomorphism group but non Jordan
diffeomorphism group: in fact, $\Diff(T^2\times S^2)$ was the
first given example of a non Jordan diffeomorphism group of a
compact manifold, see \cite{CPS}.

In this paper we only consider manifolds without boundary, so
for us {\it compact manifold} is synonym of {\it closed
manifold}.

\subsection{Finite subgroups of $\Ham(X,\omega)$}
This is our first main result.

\begin{theorem}
\label{thm:main-ham} Let $(X,\omega)$ be a $2n$-dimensional
compact and connected symplectic manifold. Then
$\Ham(X,\omega)$ is $(C,n)$-Jordan for some $C$ depending only
on $H^*(X)$.
\end{theorem}

Here $H^*(X)$ denotes the integral cohomology of $X$.
In particular, $C$ is independent of $\omega$. As mentioned earlier, Theorem
\ref{thm:main-ham} is not a theorem on smooth finite
transformation groups in symplectic disguise. Indeed, there are
plenty of examples of compact symplectic manifolds whose
diffeomorphism group is not Jordan, and in many cases even the
identity component of the diffeomorphism group fails to be
Jordan. So we can not replace
$\Ham(X,\omega)$ by $\Diff_0(X)$.
In fact, we will see below
that one can neither replace $\Ham(X,\omega)$ by $\Symp_0(X,\omega)$.

For example, if $Y$ is any compact smooth manifold supporting an
effective action of $\SU(2)$ or $\SO(3,\RR)$, then
$\Diff_0(T^2\times Y)$
is not Jordan \cite{CPS,M4}; more
concretely, for any prime $p$ there is a subgroup of
$\Diff_0(T^2\times Y)$ which is
isomorphic to the Heisenberg group over $\FF_p$.\footnote{Actually the results in \cite{M4} refer to
the full diffeomorphism group, but it is easy to check that all
group actions that are defined there give rise to finite
subgroups of the identity component of $\Diff$.}
This implies that if $Y$ is a compact symplectic
manifold supporting an effective smooth action of $\SU(2)$ or
$\SO(3,\RR)$ then, denoting by $\Omega$ the set of all symplectic forms on $T^2\times Y$,
$$\dD(T^2\times Y)\setminus\bigcup_{\omega\in\Omega}\hH(T^2\times Y,\omega)\quad\text{ contains infinitely many elements.}$$

For some symplectic manifolds $(X,\omega)$ Theorem
\ref{thm:main-ham}  implies that $\Symp_0(X,\omega)$
is Jordan:

\begin{corollary}
Let $\dim X=2n$. If the flux homomorphism
$$\pi_1(\Symp_0(X,\omega))\to H^1(X;\RR)$$ vanishes, then
$\Symp_0(M,\omega)$ is $(C,n)$-Jordan for some $C$ depending
only on $H^*(X;\ZZ)$.
\end{corollary}

See \cite[\S 10.2]{McDS} for the definition of the flux
homomorphism.

\begin{proof}
If $\Gamma$ denotes the image of the flux map
$\pi_1(\Symp_0(X,\omega))\to H^1(X;\RR)$ then there is an exact
sequence
$$\Ham(X,\omega)\to \Symp_0(X,\omega)\stackrel{F}{\longrightarrow} H^1(X;\RR)/\Gamma,$$
where $F$ denotes the flux (see \cite[Corollary 10.18]{McDS}).
If $\Gamma=0$ then any finite subgroup $G\subset
\Symp_0(X,\omega)$ satisfies $F(G)=0$, so
$G\subset\Ham(X,\omega)$.
\end{proof}

There exist many examples of manifolds $(X,\omega)$
with nonzero first Betti number  for which the flux homomorphism
$\pi_1(\Symp_0(X,\omega))\to H^1(X;\RR)$ vanishes, see e.g.
\cite{KKM}.

\subsection{Finite subgroups of $\Symp(X,\omega)$}
It is natural to wonder to what extent Theorem \ref{thm:main-ham} is true if one replaces $\Ham(X,\omega)$ by the entire group of symplectomorphisms of $X$. The answer turns out to be yes in some cases.

\begin{theorem}
\label{thm:main-sympl} Let $(X,\omega)$ be a $2n$-dimensional
compact and connected symplectic manifold satisfying
$b_1(X)=0$. Then $\Symp(X,\omega)$ is $(C,n)$-Jordan for some
$C$ which only depends on $H^*(X)$.
\end{theorem}

The situation for manifolds $(X,\omega)$ with $b_1(X)\neq 0$ is
more involved. The case $X=T^2\times S^2$ has been studied in
detail in \cite{M5}, and it turns out that for any symplectic
form $\omega$ on $T^2\times S^2$ the group of
symplectomorphisms $\Symp(T^2\times S^2,\omega)$ is $C$-Jordan
for some $C=C(\omega)$ depending on $\omega$, but there is no
upper bound for $C(\omega)$ as $\omega$ moves along the set of
symplectic forms on $T^2\times S^2$ . This is related to the
fact that $\Diff(T^2\times S^2)$ is not Jordan, and a similar
phenomenon occurs for any product $T^2\times Y$ with $Y$
symplectic and supporting a Hamiltonian action of $\SU(2)$:
namely, for any prime $p$ there is a symplectic form $\omega$
on $T^2\times Y$ such that $\Symp(T^2\times Y,\omega)$ contains
a subgroup isomorphic to the Heisenberg group over $\FF_p$
(this is a straightforward generalization of \cite[\S 3]{M5}).

Hence, the hypothesis $b_1(X)=0$ cannot be removed in Theorem
\ref{thm:main-sympl}, and from this perspective Theorem
\ref{thm:main-ham} is optimal. On the other hand, the author
does not know any example of compact symplectic manifold whose
symplectomorphism group fails to be Jordan (compactness is
crucial here, since clearly the symplectomorphism group of the
cotangent bundle of $T^2\times S^2$ is not Jordan). It is also
apparently unknown at present whether there exists a simply
connected compact smooth manifold whose diffeomorphism group is
not Jordan.

We actually prove a stronger result than Theorem
\ref{thm:main-sympl}. To state it, the following notation will
be useful: if $X$ is a smooth manifold and $J$ is an almost
complex structure on $X$ then we denote by
$$\Diff(X,J)\subset\Diff(X)$$
the group of all diffeomorphisms of $X$ preserving $J$.

\begin{theorem}
\label{thm:main-ac-str} Let $(X,J)$ be a $2n$-dimensional
almost complex, compact and connected smooth manifold
satisfying $b_1(X)=0$, and assume that there exists $\omega\in
H^2(X;\RR)$ such that $\omega^n\neq 0$. Then $\Diff(X,J)$ is
$(C,n)$-Jordan for a constant $C$ depending only on $H^*(X)$.
\end{theorem}

In particular, the constant $C$ is independent of $J$ and $\omega$
(and $J$ and $\omega$ need not be related in any way).
Theorem \ref{thm:main-ac-str} implies a particular case of a recent
result of Meng and Zhang \cite{MZ} stating that automorphism
groups of (not necessarily smooth)
projective varieties over any algebraically closed field of
characteristic zero are Jordan.

\begin{remark}
\label{rmk:ac-structures}
Theorem \ref{thm:main-sympl} follows from Theorem
\ref{thm:main-ac-str} because for any symplectic manifold $(X,\omega)$ and any
symplectic action of a compact Lie group $K$ on $X$ there exists a $K$-invariant almost complex structure on $X$ (see e.g. \cite[Lemma 5.49]{McDS} and the remark before it).
\end{remark}

Although we do not know at present whether the symplectomorphism group of
every compact symplectic manifold is Jordan, we can prove the following weaker
statement.

\begin{theorem}
\label{thm:2-step-nilpotent}
Let $(X,\omega)$ be a compact symplectic manifold. There exists a constant $C$ depending only
on $H^*(X)$
with the property that any finite subgroup $\Gamma\subset\Symp(X,\omega)$ has a
subgroup $N\subseteq \Gamma$ satisfying $[\Gamma:N]\leq C$
and $N$ is either abelian or $2$-step nilpotent.
\end{theorem}

Recall that a group $N$ is $2$-step nilpotent (equivalently, $N$ has nilpotency class $2$)
if $N$ is not abelian and every
three elements $a,b,c\in N$ satisfy $[[a,b],c]=1$.

\subsection{Fixed points of finite symplectic $p$-group actions}
To prove Theorems \ref{thm:main-ham} and \ref{thm:main-ac-str} we will use the general criterion of Jordanness given in \cite{MT}. This will allow us to restrict attention to fixed point properties of finite $p$-group actions. The theorems will then be a consequence of the following results.

\begin{theorem}
\label{thm:fixed-pt-hamiltonian} Let $(X,\omega)$ be a compact
symplectic manifold. There exists a constant $C$, depending only on $H^*(X)$, such that
for any prime $p$ and any finite $p$-subgroup $G\subset\Ham(X,\omega)$ there is a subgroup
$G_0\subseteq G$ satisfying $[G:G_0]\leq C$ and
$X^{G_0}\neq\emptyset$.
\end{theorem}

In particular, if $p>C$ then any finite $p$-subgroup
$G\subset\Ham(X,\omega)$ satisfies $X^G\neq\emptyset$, and
since the differential of the action at any $x\in X^G$ embeds
$G$ inside $\GL(T_xM)$ (see e.g. the proof of Corollary
\ref{cor:Jordan-classic}) we may combine this with Jordan's
theorem (see Corollary \ref{cor:Jordan-classic}) to obtain the
following result:

\begin{corollary}
Let $(X,\omega)$ be a $2n$-dimensional compact symplectic manifold.
There exists a constant $C'$, depending only on $H^*(X)$, such
that for any prime $p>C'$ and any finite $p$-group
$G\subset\Ham(X,\omega)$ we have: $X^G\neq\emptyset$, $G$ is
abelian, and $G$ can be generated by $n$ (or fewer) elements.
\end{corollary}

This has the interesting consequence that for many compact
symplectic manifolds $(X,\omega)$ such that $\pi_2(X)\neq 0$
(hence, not satisfying the hypothesis of the theorem of Polterovich which was mentioned earlier) the difference
$$\sS(X,\omega)\setminus \hH(X,\omega)$$
contains infinitely many elements. For example, if
$X=T^{2n}\times S^2$ is endowed with a product symplectic form
$\omega$ and the restriction of $\omega$ to $T^{2n}$ is
translation invariant then for any prime $p$ there is a
subgroup of $\Symp_0(X,\omega)$ isomorphic to $(\ZZ/p)^{2n+1}$,
whereas if $p$ is big enough then any subgroup of
$\Ham(X,\omega)$ isomorphic to $(\ZZ/p)^r$ satisfies $r\leq
n+1$.

\begin{theorem}
\label{thm:fixed-pt-symplectic} Let $(X,J)$ be a
$2n$-dimensional almost complex, compact and connected smooth
manifold satisfying $b_1(X)=0$, and assume that there exists
$\omega\in H^2(X;\RR)$ such that $\omega^n\neq 0$. There exists
a constant $C$, depending only on $H^*(X)$, with the following
property. Let $p$ be any prime, and let $G\subset\Diff(X,J)$ be
a finite $p$-subgroup. Then there is a subgroup $G_0\subseteq
G$ satisfying $[G:G_0]\leq C$ and $X^{G_0}\neq\emptyset$.
\end{theorem}

Theorem \ref{thm:fixed-pt-symplectic} solves a weaker version of a  classical problem in
transformation groups, namely that of proving existence of
fixed points of finite $p$-group actions on projective
varieties. There exist many partial results on this question,
assuming different restrictions on the group, the manifold, or
the action, see e.g. \cite[Theorem (1.5)]{Br1}, \cite[Theorem
(4.10)]{Br2}, or more recently \cite[Theorem 4.2]{PS2} combined
with \cite{Bi}.

Combining the previous theorem with Jordan's theorem (see
Corollary \ref{cor:Jordan-classic}) we obtain.

\begin{corollary}
Let $(X,J)$ be a $2n$-dimensional almost complex, compact and
connected smooth manifold satisfying $b_1(X)=0$, and assume
that there exists $\omega\in H^2(X;\RR)$ such that
$\omega^n\neq 0$. There exists a constant $C$, depending only
on $H^*(X)$,  such that for any prime
    $p>C$, any finite $p$-subgroup $G\subset\Diff(X,J)$ satisfies $X^G\neq\emptyset$, $G$ is
    abelian, and $G$ can be generated by $n$ (or fewer) elements.
\end{corollary}

\subsection{Actions of finite $p$-groups on vector bundles}
The main technical tool developed in this paper, from which we
will deduce Theorems \ref{thm:fixed-pt-hamiltonian} and
\ref{thm:fixed-pt-symplectic}, is a fixed point theorem for
actions of finite $p$-groups on complex vector bundles.

Before explaining the theorem, let us recall some standard
terminology to avoid potential confusions. Let $G$ be a group
acting smoothly on (the total space of) a vector bundle $E\to
X$. We say that the action is by vector bundle automorphisms if
it sends fibers to fibers and the action is compatible with the
vector space structure on the fibers. In particular, there is
an induced action of $G$ on $X$ such that for any $g\in G$,
$x\in X$ and $e\in E_x$ we have $g\cdot e\in E_{g\cdot x}$. We
say that the action of $G$ on $E$ lifts the action on $X$.

If $E\to X$ is a complex vector bundle, with $X$ a smooth
manifold, and $J$ is an almost complex structure on $X$, we
denote by
$$\Aut(E,X,J)$$
the group of all vector bundle automorphisms of $E\to X$
lifting elements of $\Diff(X,J)$.

\begin{theorem}
\label{thm:main-fixed-point} Let $(X,J)$ be an almost complex,
compact and connected smooth manifold and let $E\to X$ be a
complex vector bundle satisfying $\rk_{\RR}(E)=\dim_{\RR} X$.
Suppose that $\la [X],\e(E)\ra\neq 0$, where $\e(E)$ is the
Euler class of $E$ and $X$ is oriented by $J$. Then there
exists a constant $C$, depending only on $\la[X],\e(E)\ra$ and
$H^*(X)$, with the following property.
Let $p$ be any prime, and let $G$ be a finite $p$-subgroup of
$\Aut(E,X,J)$. Then there exists an abelian subgroup
$A\subseteq G$ satisfying $[G:A]\leq C$ and $X^A\neq\emptyset$.
\end{theorem}

The proof of Theorem \ref{thm:main-fixed-point} proceeds in two
steps. First we assume that $G$ has an abelian normal subgroup
$B$ such that $G/B$ is cyclic, and then the result is extended
to arbitrary finite $p$-groups by combining that particular
case with Jordan's theorem. Of these two steps, the first one
is substantially more involved than the second. To briefly
explain this, let us consider the situation in which $G$ is
abelian, since  in this case the main difficulties are already
present.

If $p$ does not divide $\la[X],\e(E)\ra$ then the
proof is similar to that of the main
theorem in \cite{Br3}. We proceed by induction on $|G|$.
Let $H\subseteq G$ be a subgroup isomorphic to $\ZZ/p$.
Since $\la[X],\e(E)\ra$ is
not divisible by $p$,
$X^H$ is nonempty. If $H=G$ then we are done, otherwise we consider the
induced action of $G/H$ on $X^H$. To be able to apply induction
it is necessary to generalize the statement to include several vector bundles
instead of only one and to replace the
Euler class by a product of Chern classes of the vector bundles (and then the
numerical hypothesis is that such product can be chosen so that its pairing with $[X]$
is not divisible by $p$).
In the induction process we add the normal bundle $N$ of $X^H\hookrightarrow X$
to the collection of vector bundles, and all vector
bundles (the {\it old ones} and $N$, the {\it new one}) have to be decomposed and
twisted by characters of $H$ so that the action of $G/H$ on $X$ lifts to an action on each of
them. Here we need all involved vector bundles (including $N$)
to carry an invariant complex structure.
This is the only moment where we use our assumption that the action of $G$ preserves an
almost complex structure $J$ on $X$: the complex structure on $N$ is taken to
be the restriction of $J$.
Finally, to guarantee that the numerical hypothesis is inherited by the action of $G/H$ on
$X^H$ and the new collection of vector bundles we use the standard fact that
the equivariant Euler class of $N$ with $\ZZ/p$ coefficients is invertible up to
inverting a generator of $H^*(BG)$.

If $p$ divides $\la[X],\e(E)\ra$ things are more involved. Suppose that $p^k$ is the smallest power of $p$ that does not divide $\la[X],\e(E)\ra$ and that $k\geq 2$. Let $H\subseteq G$ be a subgroup isomorphic to $\ZZ/p^k$
(if such $H$ does not exist, then $|G|$ is smaller than a constant depending on $X$ and $k$,
so we are done). Now it is not necessarily true that $X^H$ is nonempty, but certainly
$X^{H'}\neq\emptyset$, where $H'=p^{k-1}H\simeq\ZZ/p$.
Let $N$ be the normal bundle of $X^{H'}\subset X$. Then $N$ is a
$G$-equivariant vector bundle, and the main  step in the proof of Theorem \ref{thm:main-fixed-point} is the proof that its equivariant Euler class
with $\ZZ/p^k$ coefficients is invertible up to inverting a generator of $H^*(BG)$, see Theorem \ref{thm:inverting-euler-class}. We use for that a well
known nilpotence argument in equivariant cohomology due to Quillen (see \cite[\S 3]{Q} --- we only
need the compact case, whose proof is elementary).
The remaining details are the same as in the case where $p$ does not divide $\la[X],\e(E)\ra$.

\begin{remark}
\label{rmk:p-odd-1} We mentioned above that the only place in
the proof of Theorem \ref{thm:main-fixed-point} where we need
that the action of $G$ preserves an almost complex structure on
$X$ is to guarantee the existence of a $G_1$-invariant complex
structure on the normal bundles of inclusions of the type
$X^{G_1}\subset X^{G_0}$, where $G_1$ is abelian, $G_0\subseteq
G_1$ and $G_1/G_0\simeq\ZZ/p$. If $p$ is odd such complex
structures always exist (see e.g. \cite[Theorem (38.3)]{CF}),
so if we only consider odd primes Theorem
\ref{thm:main-fixed-point} is true without assuming the
existence of an almost complex structure (see \cite{Br3} for
similar considerations).
\end{remark}

To deduce Theorems \ref{thm:2-step-nilpotent}, \ref{thm:fixed-pt-symplectic} and
\ref{thm:fixed-pt-hamiltonian} from Theorem
\ref{thm:main-fixed-point} we apply the following result.

\begin{theorem}
\label{thm:liftings} Let $X$ be compact smooth manifold,
and let $G\subset\Diff(X)$ be a
finite subgroup. Let $L\to X$ be a complex line bundle
satisfying $g^*L\simeq L$ for every $g\in G$. Suppose that one
of the following conditions holds true.
\begin{enumerate}
\item $b_1(X)=0$, or
\item there exists a symplectic form $\omega$ on $X$ such that
$G\subset\Ham(X,\omega)$, or
\item there exists a finite group $\Gamma\subset\Diff(X)$, whose
action on $H^1(X)$ is trivial, which satisfies
$G\subseteq[\Gamma,\Gamma]$ and $\gamma^*L\simeq L$ for every $\gamma\in\Gamma$.
\end{enumerate}
Then there exists a finite group $G'$ sitting in a short exact
sequence
$$1\to H\to G'\stackrel{\pi}{\longrightarrow} G\to 1,$$
where $H$ is finite cyclic and $|H|$ divides $|G|$, and a smooth action of $G'$ on $L$ by
bundle automorphisms lifting the action of $G$ on $X$. The latter means
that for any $\gamma\in G'$, any $x\in X$, and any
$\lambda\in L_x$ we have $\gamma\cdot\lambda\in
L_{\pi(\gamma)(x)}.$
\end{theorem}

The proof of Theorem \ref{thm:liftings} under hypothesis (2)
(see Theorem
\ref{thm:lift-Ham-action-line-bundle}) uses a result of
Lalonde, McDuff and Polterovich \cite{LMP,McD} based on the
Seidel representation and implying the topological rigidity of
Hamiltonian loops. In contrast, the proof under hypothesis (1) or (3) (see Theorems
\ref{thm:lift-action-line-bundle} and \ref{thm:lift-commutator-action-line-bundle})
only uses elementary cohomological arguments.

\begin{remark}
\label{rmk:p-odd-2} Since Theorem \ref{thm:liftings} does not
involve any almost complex structure, it follows from Remark
\ref{rmk:p-odd-1} that if we only consider odd primes then
Theorem \ref{thm:fixed-pt-symplectic} is true for finite
$p$-subgroups of $\Diff(X)$ (so there is no need to consider
almost complex structures).
\end{remark}

\subsection{Contents}
Section \ref{s:equivariant-cohomology} reviews some material on
equivariant cohomology, introduces the notion of
generic cohomology class of a cyclic $p$-group, and proves
(Theorem \ref{thm:inverting-euler-class}) the invertibility of
the Euler class of normal bundles of partially fixed point
submanifolds of smooth $\ZZ/p^k$ actions (see the comments
after the statement of Theorem \ref{thm:main-fixed-point}). In
Section \ref{s:localization} we prove a localization theorem
for smooth actions of $\ZZ/p^k$ on manifolds. This will be the
main building block of the proof of Theorem
\ref{thm:main-fixed-point}. Other ingredients will be a few
technical results on finite $p$-groups proved in Section
\ref{s:finite-p-groups}. The proof of Theorem
\ref{thm:main-fixed-point} is given in Section
\ref{s:proof-thm:main-fixed-point}. Section
\ref{s:lifting-actions-line-bundle} contains the proof of
Theorem \ref{thm:liftings}. Finally, all theorems stated in this
introduction except for Theorems \ref{thm:main-fixed-point} and
\ref{thm:liftings} are proved in Section \ref{s:proofs-main}.

\subsection{Conventions}
By a natural number we understand a strictly positive integer.

Whenever we say that a group $G$ can be generated by $d$
elements we mean that there is a generating set
$\{g_1,\dots,g_d\}$ for $G$, where the $g_j$'s need not be
distinct.
If $G$ is a group and $S_1,\dots,S_r$ are subsets of $G$, we
denote by $\la S_1,\dots,S_r\ra$ the subgroup of $G$ generated
by the elements in the union of the $S_j$'s. Abusing notation,
if a set has a unique element $s$, we will sometimes denote the set using
the same symbol $s$ instead of $\{s\}$.
If $A$ is an abelian group,
$\wh{A}:=\Hom(A,\CC^*)$ denotes the group of characters of $A$.

As mentioned previously, we only consider manifolds without boundary.
Manifolds are not necessarily supposed to be connected, and
vector bundles over disconnected manifolds need not have
constant rank.
Group actions on smooth manifolds will always be assumed to be smooth.

All cohomology groups will be, unless otherwise specified, with integer coefficients.

\subsection{Acknowledgements}
I wish to thank A. Jaikin, A. Turull and C. S\'aez for useful
comments. Special thanks to A. Jaikin for providing the proof
of Lemma \ref{lemma:alpha-nilpotent-2}, which is much shorter
and more efficient than the original one. Many thanks finally
to the referee for a detailed and very useful report, for a
number of corrections and suggestions to improve the paper, and
for providing an alternative and more direct proof of Theorem
\ref{thm:lift-action-line-bundle}.

\section{Equivariant cohomology}
\label{s:equivariant-cohomology}

This section contains the results from equivariant cohomology
that will be used later in the localization arguments. Let us
first briefly recall the basic definitions (see e.g.
\cite[Chap. III, \S1]{tD} for details). If $G$ is a topological
group and $X$ is a topological space acted on continuously by
$G$ the $G$-equivariant cohomology of $X$ is by definition
$$H_G^*(X)=H^*(EG\times_GX),$$
where $EG\to BG$ is the universal $G$-principal bundle. The
space $X_G:=EG\times_GX$ is called the Borel construction, and
its natural projection $X_G\to BG$ endows $H_G^*(X)$ with the
structure of a module over $H^*(BG)$.

The two extreme cases are $G$ acting freely on $X$ (in this
case $X_G\simeq X/G$, at least if the action admits slices)
and $G$ acting trivially on $X$
(in this case $X_G\simeq BG\times X$). In
particular, if $X$ is a $G$-space consisting of a unique point
then $H_G^*(X)\simeq H^*(BG)$. For a
general $G$-space the equivariant cohomology can be used to
obtain information on isotropy groups. In our case $G$ will
always be a finite group, and for this reason it will be
essential to work with integer coefficients: rational and real
coefficients are of no use in this situation, since if $G$ is
finite then $H_G^*(X;\QQ)\simeq H^*(X/G;\QQ)$.

If $V\to X$ is a vector bundle and the action of $G$ on $X$
lifts to an action on $V$ by vector bundle automorphisms, then
$V_G$ is in a natural way a vector bundle over $X_G$. The
equivariant characteristic classes of $V$ are the ordinary
characteristic classes of $V_G$. So if $V$ is a complex (resp.
real oriented) vector bundle then the equivariant Chern classes
(resp. Euler class) of $V$ are $c_j^G(V):=c_j(V_G)$ (resp.
$\e^G(V):=\e(V_G)$).

\subsection{The pushforward map}
\label{s:umkehrungs}
\newcommand{\exc}{\operatorname{exc}}

Of fundamental importance in the localization arguments that we
are later going to use is the so-called
pushforward\footnote{There is no consensus in the literature on
how to name this notion. {\it Pushforward map} seems to be the
most usual name in the recent literature on equivariant
cohomology in symplectic geometry, see e.g. \cite{GGK,GS}. {\it
Umkehr/Umkehrung} (the German word for "reversal") was the name
used by Hopf \cite{Ho} in the first paper on the subject (in the
non equivariant context), and
it is still used in homotopy theory \cite{CK}. Atiyah and Bott
use it in their classical paper \cite{AB} on equivariant
cohomology in symplectic geometry. For ordinary non-equivariant
cohomology, one also uses {\it shriek} or {\it transfer map}
\cite[Chap. VI, Def. 11.2]{Br}. However, in equivariant
cohomology  {\it transfer map} usually means something
different, see e.g. \cite{Bo,Bre2,tD}.} map in
equivariant cohomology. Let $M$ and $N$ be compact smooth
oriented manifolds endowed with orientation preserving actions
of a compact Lie group $G$, and let $f:M\to N$ be a
$G$-equivariant smooth map. The pushforward is a map
$$f^G_*:H_G^*(M)\to H_G^{*+\dim N-\dim M}(N)$$
which enjoys the following properties:
\begin{itemize}
\item[(P1)] (composition) if $g:M\to R$ is another
    $G$-equivariant map, with $R$ a compact smooth oriented
    $G$-manifold, then $(g\circ f)_*^G=g^G_*\circ f^G_*$;
\item[(P2)] (product formula) for any $\alpha\in H_G^*(N)$
    and any $\beta\in H_G^*(M)$ we have
$$f^G_*((f^*\alpha)\beta)=\alpha(f^G_*\beta);$$
combining (P1) and (P2) it follows that $f^G_*$ is a
morphism of $H^*(BG)$-modules;
\item[(P3)] (embeddings) if $f$ is an embedding, then
    $f^G_*$ factors through the Thom isomorphism of the
    normal bundle $\nu\to M$; more precisely, $f^G_*$ is
    equal to the composition
\begin{multline*}
H_G^*(M)\stackrel{T}{\longrightarrow}
H_G^*(\nu,\nu\setminus\nu_0)\stackrel{\varepsilon}{\longrightarrow}
H_G^*(D(\nu),D(\nu)\setminus\nu_0)\longrightarrow \\
\stackrel{(\eta^*)^{-1}}{\longrightarrow}
H_G^*(\exp(D(\nu)),\exp(D(\nu))\setminus M)
\stackrel{\varepsilon}{\longrightarrow} H_G^*(N,N\setminus
M)\to H_G^*(N),
\end{multline*}
where $\nu_0\subset\nu$ is the zero section, $T$ is Thom
isomorphism, $D(\nu)\subset\nu$ is the open unit disk
bundle defined with respect to a $G$-invariant Riemannian
metric $h$ on $N$, $\eta:D(\nu)\to N$ is the exponential
map w.r.t. $h$ (which we assume to have injectivity radius
at least $1$, so that $D(\nu)$ is a tubular neighborhood of
$M$ --- otherwise we rescale $D(\nu)$ in the fiberwise
directions), $\varepsilon$ denotes excision, and the last
map is the natural map in cohomology; since
all maps except the last one are isomorphisms,
the image of $f^G_*$ can be
identified with the kernel of the restriction map
$H^*_G(N)\to H_G^*(N\setminus M)$; another consequence is
the formula
$$f^*f^G_*(\alpha)=\alpha\e^G(\nu),$$
where $\e^G$ is the equivariant Euler class;
\item[(P4)] (functoriality) let $\rho:K\to G$ be a
    morphism of groups; using $\rho$ and the $G$-action, we
    may define an action of $K$ on $M$; let $\pi:M\to\{*\}$
    be the map to a point; let $\rho_M^*:H^*_G(M)\to
    H^*_K(M)$ and $\rho^*:H^*(BG)\to H^*(BK)$ be the
    natural maps induced by $\rho$; then we have
    $$\rho^*\pi^G_*=\pi^K_*\rho_M^*.$$
\end{itemize}

In the context of symplectic geometry these properties and
their application to localization arguments are well known, but
they are typically applied to compact connected Lie group
actions, where real coefficients are enough (see e.g.
\cite{AB,GGK,GS}). The proof of (P1)---(P4) for real
coefficients is easily obtained using the Cartan--Weil model
for equivariant cohomology \cite[\S 10.7]{GS}.

Since for us the use of integral coefficients is crucial, we
briefly sketch for the reader's convenience the definition of
$f^G_*$ over $\ZZ$. Consider first the case $G=\{1\}$. Then
$f^G_*$ can be defined as $D_N^{-1}\circ f_*\circ D_M$, where
$D_M:H^*(M)\to H_{\dim M-*}(M)$ and $D_N:H^*(N)\to H_{\dim
N-*}(N)$ are the Poincar\'e duality maps and $f_*:H_*(M)\to
H_*(N)$ is the map induced in homology by $f$ (see \cite[Chap.
VI, Def. 11.2]{Br}). To define $f_*^G$ for arbitrary finite $G$
one can use finite dimensional approximations of the
classifying space $BG$. More precisely, in order to define
$f_r^G:H_G^r(M)\to H^{r+\dim N-\dim M}(N)$ one considers a
principal $G$-bundle $P\to B$ with $B$ a smooth compact,
connected and oriented manifold, such that $P$ is $k$-connected
for  $k$ big enough so that the classifying map $B\to BG$ for
$P$ induces isomorphisms
$$H^r(P\times_GM)\simeq H_G^r(M),\qquad H^{r+\dim N-\dim
M}(P\times_GN)\simeq H_G^{r+\dim N-\dim
M}(N)$$ (one can take $P$ to be a
Stiefel manifold, see e.g. \cite[Example C.1]{GGK}).
Combining these isomorphisms with the map
$$D^{-1}\circ f_*\circ D:H^r(P\times_GM)\to H^{r+\dim N-\dim M}(P\times_GN)$$
one obtains $f_r^G$ (here $D$ is
Poincar\'e duality and $f_*:H_*(P\times_GM)\to
H_*(P\times_GN)$ is the map induced by $f$). The
following lemma, whose proof is left to the reader (see
\cite[Chap. VI]{Br} for the necessary tools), can be used to
prove that the resulting map $f_*^G$ is independent of the
choice of $P\to B$.

\begin{lemma}
\label{lemma:square-umkehrungs} Consider a commutative diagram
\begin{equation}
\label{eq:diagrama-umkehr}
\xymatrix{X\ar[r]^i\ar[d]_f & Y \ar[d]^g \\
U \ar[r]_h & V}
\end{equation}
where $U,V,X,Y$ are compact, connected and oriented smooth
manifolds and $f,g,h,i$ are smooth maps. Suppose that $i$ and
$h$ are embeddings and the normal bundles $\nu_h\to U$ and
$\nu_i\to X$ satisfy $\nu_i\simeq f^*\nu_h$. Then we have
$$D_U f_* D_X i^*=h^*D_Vg_* D_Y$$ as maps from $H^*(Y)$ to
$H^*(U)$, where $D_M$ denotes Poincar\'e duality on the
manifold $M$.
\end{lemma}

To apply the previous lemma, note that for any two principal $G$-bundles $P_0\to B_0$ and
$P_1\to B_1$ one can find another principal $G$-bundle $P'\to
B'$ admitting $G$-equivariant embeddings $P_i\to P'$ ($i=0,1$)
in such a way that the resulting diagrams (with $f_i$ and $f'$
the maps induced by $f$)
$$\xymatrix{P_i\times_GM \ar[r]^i\ar[d]_{f_i} & P'\times_GM \ar[d]^{f'} \\
P_i\times_G N \ar[r]_h & P'\times_GN}
$$
satisfy the hypothesis of Lemma \ref{lemma:square-umkehrungs}.
This immediately implies that $f_*^G$ is well defined.

The proof that the pushforward map satisfies properties
(P1)--(P3) is straightforward using standard results as in
\cite[Chap. VI]{Br}. To prove (P4) one may use a variant of
Lemma \ref{lemma:square-umkehrungs} in which $f$ and $g$ are
assumed to be submersions, the square
(\ref{eq:diagrama-umkehr}) is Cartesian, and the maps $i$ and
$f$ can be identified with the projections $U\times_VY\to Y$
and $U\times_VY\to U$.

\subsection{Cohomology of cyclic $p$-groups}

Let us identify the circle $S^1$ with the complex numbers of modulus $1$, and let
$$\mu_{p^k}\subset S^1$$
be the group of $p^k$-th roots of unity, where $p$ is a prime
and $k$ a natural number. The following lemma is standard
(see e.g. \cite[Chap. III, \S 2]{tD}).
\begin{lemma}
\label{lemma:roots-unity-in-S}
Let $t=c_1(ES^1)\in H^2(BS^1)$.
We have $H^*(BS^1)=\ZZ[t]$. The map
$$\iota_k^*:H^*(BS^1)=\ZZ[t]\to H^*(B\mu_{p^k})$$ induced
by the inclusion $\iota_k:\mu_{p^k}\hookrightarrow S^1$ is
surjective and its kernel is the ideal of $H^*(BS^1)$
generated by $p^kt$. Furthermore, $\iota_k^*t=c_1(L)$, where
$L=E\mu_{p^k}\times_{\iota_k}\CC$.
\end{lemma}

\begin{lemma}
\label{lemma:roots-unity-into-roots-unity} Let
$j:\mu_{p^{k-1}}\hookrightarrow\mu_{p^{k}}$ be the inclusion.
For any nonnegative integer $r$ there
is a commutative diagram with exact rows
$$\xymatrix{p^{k}\ZZ\la t^r\ra\ar@{^{(}->}[r] &
\ZZ\la t^r\ra=H^{2r}(BS^1)\ar@{=}[d]\ar[r]^-{\iota_{k}^*} & H^{2r}(B\mu_{p^{k}})\ar[d]^{j^*} \ar[r] & 0 \\
p^{k-1}\ZZ\la t^r\ra\ar@{^{(}->}[r] & \ZZ\la t^r\ra=H^{2r}(BS^1)\ar[r]^-{\iota_{k-1}^*} &
H^{2r}(B\mu_{p^{k-1}}) \ar[r] & 0 .}$$
\end{lemma}
\begin{proof}
The exactness of the rows follows from Lemma \ref{lemma:roots-unity-in-S}.
The commutativity of the square is a consequence of the identity
$\iota_k\circ j=\iota_{k-1}$.
\end{proof}

\subsection{Generic cohomology classes of cyclic $p$-groups}
\label{ss:generic-classes} Let $k$ be a natural number and let
\begin{equation}
\label{eq:def-tau-k}
\tau_k=\iota_k^*t\in H^2(B\mu_{p^k}),
\end{equation}
where $\iota_k:\mu_{p^k}\hookrightarrow S^1$ is the inclusion. We say that a cohomology class
$\beta\in H^*(B\mu_{p^k})$ is {\bf generic} if for any natural
$r$ the class $\tau_k^r \beta$ is nonzero.

It follows from Lemma \ref{lemma:roots-unity-in-S}
that if $\beta\in H^d(B\mu_{p^k})$ then
\begin{equation}
\label{eq:charac-generic}
\beta\text{ is generic }\Longleftrightarrow \left\{
\begin{array}{ll}
\beta\text{ is not divisible by $p^k$} & \quad\text{if $d=0$,}\\
\beta\neq 0 & \quad\text{if $d>0$.}\end{array}\right.
\end{equation}
The following two properties are obvious and will be used
repeatedly in our arguments:
\begin{enumerate}
\item if $\alpha_1,\dots,\alpha_n\in H^*(B\mu_{p^k})$
    and $\alpha_1+\dots+\alpha_n$ is generic, then at least
    one of the $\alpha_j$'s is generic;
\item if $\alpha\in H^*(B\mu_{p^k})$ and $r$ is a nonnegative integer,
    then $\alpha$ is generic if and only if
    $\tau_k^r\alpha$ is generic.
\end{enumerate}

\begin{lemma}
\label{lemma:free-action-roots-unity} Let $X$ be a smooth
oriented compact manifold endowed with an orientation
preserving action of $\mu_{p^k}$. Let $\pi:X\to\{*\}$ be the
map to a point. If for some $\alpha\in H^r_{\mu_{p^k}}(X)$ the
class $\pi^{\mu_{p^k}}_*\alpha\in H^*(B{\mu_{p^k}})$ is
generic, then the action of $\mu_{p^k}$ on $X$ is not free.
\end{lemma}
\begin{proof}
Denote for convenience $G=\mu_{p^k}$. By the product formula
((P2) in  \S\ref{s:umkehrungs}) we have
$\pi^G_*((\pi^*\tau_k^s)\alpha)=\tau_k^s \pi^G_*\alpha$ for any
natural $s$ and any $\alpha\in H^*_{G}(X)$. Suppose that ${G}$
acts freely on $X$. Then the Borel construction $X_G$ is
homotopy equivalent to the smooth manifold $X/G$. If $s+r$ is
bigger than the dimension of each connected component of $X/G$
then for any $\alpha$ we have
$$\tau_k^s\alpha\in H_{G}^{s+r}(X)=H^{s+r}(X/G)=0,$$ so
$\tau_k^s\pi^G_*\alpha=\pi^G_*(\tau_k^s\alpha)=0$. Hence
$\pi^G_*\alpha$ is not generic.
\end{proof}

The definition of generic cohomology classes can be naturally extended
to abstract cyclic $p$-groups. Namely, if $G$ is
one such group we may take any isomorphism
$\theta:G\stackrel{\simeq}{\longrightarrow}\mu_{p^k}$ for some $k$ and
declare that a cohomology class $\beta_G\in H^*(BG)$ is generic if it
is equal to $\theta^*\beta$ for some generic class $\beta\in H^*(B\mu_{p^k})$.
By (\ref{eq:charac-generic}) this definition is independent of $\theta$.

\subsection{A nilpotence lemma}

Let $G=\mu_{p^k}$ and $G_0=\mu_{p^{k-1}}\subset G$ for some integer $k\geq 2$.
Suppose that $G$ acts smoothly on a compact smooth manifold $X$. Denote by
$$\lambda:H_G^*(X)\to H_{G_0}^*(X)$$
the restriction map.

\begin{lemma}
\label{lemma:nilpotence}
If $\alpha,\beta\in H_G^*(X)$ satisfy $\lambda(\alpha)=\lambda(\beta)$ then
$\alpha^{p^r}=\beta^{p^r}$ for some $r$.
\end{lemma}
\begin{proof}
We first prove that for any $\delta\in\Ker\lambda$ we have
$p\delta=0$. Fixing some model $EG\to BG$ for the universal
bundle of $G$, we can identify $X_{G_0}$ with
$EG\times_{G_0}X$, so that the map
$$q:X_{G_0}\to X_G=EG\times_GX$$
inducing $\lambda$ can be chosen to be a degree $p$ covering
(namely, the quotient by the residual action of $G/G_0$). This
immediately implies that any nontrivial element in the kernel of
$\lambda$ has order $p$. To prove this, let $S_*$ be the group
of singular chains, define a "trace" map $\tau_*:S_*(X_G)\to
S_*(X_{G_0})$ by sending a singular simplex $s:\Delta^r\to X_G$
to the sum of its $p$ different lifts $\Delta^r\to X_{G_0}$,
and extend the definition linearly to singular chains; then
$\tau_*$ is a morphism of complexes, and $q_*\tau_*:S_*(X_G)\to
S_*(X_G)$ is multiplication by $p$; dualising, the map
$\tau^*:S^*(X_{G_0})\to S^*(X_G)$ satisfies
$\tau^*q^*(\alpha)=p\alpha$ for every $\alpha\in S^*(X_G)$.
Now, if $\lambda([\alpha])=[q^*\alpha]=0$ then
$p[\alpha]=[\tau^*q^*\alpha]=\tau^*[q^*\alpha]=0$.

Let $\pi:X\to X/G$ be the quotient map.
Consider the sheaf $\hH_G^*$ (resp. $\hH_{G_0}^*$) of graded rings on $X/G$ defined by
$\hH_G^*(U)=H_G^*(\pi^{-1}(U))$ (resp. $\hH_{G_0}^*(U)=H_{G_0}^*(\pi^{-1}(U))$) for every open
subset $U\subseteq X/G$ (see \cite[\S 3]{Q}). We have a commutative diagram of rings
\begin{equation}
\label{eq:equivariant-cohomology-sheaf}
\xymatrix{H^*_G(X)\ar[d]_{\sigma}\ar[r]^{\lambda} & H_{G_0}^*(X) \ar[d]^{\sigma_0} \\
H^0(X/G,\hH_G^*)\ar[r]^{\mu} & H^0(X/G,\hH_{G_0}^*),}
\end{equation}
where $\mu$ is defined by restriction similarly to $\lambda$
and $\sigma,\sigma_0$ are defined (also by restriction) in
\cite[\S 3]{Q}. We next prove that for any $\beta\in
H^0(X/G,\hH_G^*)$ such that $\mu(\beta)=0$ we have $\beta^2=0$.
We need for that to describe the stalks of $\hH_G^*$ and
$\hH_{G_0}^*$ and the action of the restriction map $\hH^*_G\to
\hH_{G_0}^*$ at the level of stalks. Let $y\in X/G$ and let
$x\in\pi^{-1}(y)$ be any preimage. Let $G_1=G_x$ be the
stabiliser of $x$. Clearly, either $G_1\subseteq G_0$ or
$G_1=G$. If $G_1\subseteq G_0$ then $(\hH_G^*)_y\simeq
H^*(BG_1)$, $(\hH_{G_0}^*)_y\simeq (H^*(BG_1))^{\oplus p}$, and
the natural restriction map $\rho_y:(\hH_G^*)_y\to
(\hH_{G_0}^*)_y$ can be identified with the diagonal inclusion.
In particular, $\rho_y$ is injective in this case. If $G_1=G$
then $(\hH_G^*)_y\simeq H^*(BG)$, $(\hH_{G_0}^*)_y\simeq
H^*(BG_0)$, and $\rho_y:(\hH_G^*)_y\to (\hH_{G_0}^*)_y$ can be
identified with the map $j^*$ in Lemma
\ref{lemma:roots-unity-into-roots-unity}. The latter has the
property that for any $a\in\Ker j^*$ we have $a^2=0$. With
these observations, the claim is clear.

We next claim that if $\delta\in\Ker \lambda$ then
$\delta^{p^r}=0$ for sufficiently big $r$. By the commutativity
of (\ref{eq:equivariant-cohomology-sheaf}) if
$\delta\in\Ker\lambda$ then $\mu(\sigma(\delta))=0$, which by
the previous claim implies that
$\sigma(\delta^2)=\sigma(\delta)^2=0$. Now the claim follows
from \cite[Prop. 3.2 and Remark 3.4]{Q}.

We are now ready to prove the lemma. Suppose that
$\alpha,\beta\in H_G^*(X)$ satisfy
$\lambda(\alpha)=\lambda(\beta)$, so that
$\delta=\beta-\alpha\in\Ker\lambda$. Let $r$ be a natural
number such that $\lambda^{p^r}=0$. Then
$$\beta^{p^r} =(\alpha+\delta)^{p^r}=\alpha^{p^r}+\left(\sum_{j=1}^{p^r-1}\left(p^r\atop j\right)\alpha^j\delta^{p^r-j}\right)+\delta^{p^r}=\alpha^{p^r}+\delta^{p^r}=\alpha^{p^r},$$
because $\left(p^r\atop j\right)$ is divisible by $p$ for any
$1\leq j\leq p^r-1$.
\end{proof}

\begin{lemma}
\label{lemma:nilpotence-2}
Let $G'\subseteq G=\mu_{p^k}$ be any nontrivial subgroup. Let
$\nu:H^*_G(X)\to H^*_{G'}(X)$ be the restriction map. If
$\alpha,\beta\in H_G^*(X)$ satisfy $\nu(\alpha)=\nu(\beta)$ then
$\alpha^{p^s}=\beta^{p^s}$ for some $s$.
\end{lemma}
\begin{proof}
Combine the previous lemma with induction on $|G/G'|$.
\end{proof}

\subsection{Decomposing vector bundles according to characters}
\label{ss:decompose-equivariant-bundles}

Let $A$ be a finite abelian group acting smoothly on a manifold
$X$ and let $E\to X$ be an
$A$-equivariant complex vector bundle. Recall that
$\wh{A}$ denotes the group $\Hom(A,\CC^*)$ of characters of $A$. For
any $\xi\in\wh{A}$ we define $E^{A,\xi}$ to be the subbundle of
$E|_{X^{A}}$ consisting of those vectors on which the action of
$A$ is given by multiplication by the character $\xi$. Namely,
$$E^{A,\xi}=\{v\in E_x\mid x\in X^{A},\,\alpha\cdot v=\xi(\alpha)v\text{ for every $\alpha\in A$}\}.$$
Since $A$ is abelian, the irreducible representations of $A$
are one-dimensional. This implies
$$E|_{X^{A}}=\bigoplus_{\xi\in \wh{A}}E^{A,\xi}.$$

\subsection{Rings generated by Chern classes}
\label{ss:Chern-classes} Let $X$ be a smooth manifold endowed
with a smooth action of a finite group $\Gamma$. Let
$W_1,\dots,W_s$ be $\Gamma$-equivariant complex vector bundles
over $X$. Let $A\subseteq \Gamma$ be a normal abelian subgroup.
We denote by
$$\cC_A(W_1,\dots,W_s)\subseteq H^*(X^A)$$
the subring generated by $1$ and the Chern classes of the
vector bundles $\{W_j^{A,\xi}\}$, where $j$ and $\xi$ run over
the sets $\{1,\dots,s\}$ and $\wh{A}$ respectively. When $A$ is
trivial we usually suppress it from the notation, so we write
$\cC(W_1,\dots,W_s)\subseteq H^*(X)$.

Let $G\subseteq\Gamma$ be any subgroup contained in the
centralizer of $A$.
Then $X^A$ is $G$-invariant and the $G$-equivariant Chern
classes of the bundles $W_j^{A,\xi}$ are naturally defined. We
denote by $\cC_A^G(W_1,\dots,W_s)\subseteq H_G^*(X^A)$ the
$G$-equivariant counterpart of $\cC_A(W_1,\dots,W_s)$. Namely,
$\cC_A^G(W_1,\dots,W_s)$ is the subring generated by $1$ and
the $G$-equivariant Chern classes of the vector bundles
$\{W_j^{A,\xi}\}$, where $j$ and $\xi$ run over the sets
$\{1,\dots,s\}$ and $\wh{A}$ respectively. Again, if $A$ is
trivial we denote $\cC_A^G(W_1,\dots,W_s)$ by
$\cC^G(W_1,\dots,W_s)$.

These definitions make sense more generally if each $W_j$ is a
$\Gamma$-equivariant complex vector bundle over a
$\Gamma$-invariant submanifold $Y_j\subseteq X$ containing
$X^A$, where the submanifolds $Y_1,\dots,Y_s$ need not be equal
to $X$ and may even be different among themselves.

\subsection{Inverting the equivariant Euler class up to $\tau_k$}
\label{ss:inverting-Euler-class} Let $k$ be any natural number,
let $G=\mu_{p^k}$ and let $H=\mu_p\subseteq G$. Let $W$ be a
complex vector bundle over a compact smooth manifold $X$, and
suppose that $G$ acts complex linearly on $W$ lifting a smooth
action on $X$. Suppose that $W^H=X$, so that:
\begin{enumerate}
\item the action of $H$ on $X$ is trivial, and
\item for any $x\in X$, any $h\in H\setminus\{1\}$ and any $v\in W_x\setminus\{0\}$
 we have $h\cdot v\neq v$.
\end{enumerate}
Define $\tau_k$ as in Subsection \ref{ss:generic-classes}. We
have:

\begin{theorem}
\label{thm:inverting-euler-class}
There exists some $f\in
\cC^G_H(W)[\tau_k]\subseteq H^*_G(X)$ and some positive integer
$U$ such that $\e^G(W)f=\tau_k^U$.
\end{theorem}
\begin{proof}

We identify the group of characters of $H$ with $\ZZ/p$, by assigning to $\chi\in\ZZ/p$
the character $\rho_{\chi}:H\to\CC^*$, $\rho_{\chi}(\theta)=\theta^{\chi}$.
Define $W_{\chi}:=W^{H,\rho_{\chi}}.$
We have a
decomposition as a Whitney sum of complex vector bundles
$$W=\bigoplus_{\chi\in \ZZ/p}W_{\chi}.$$
Assumption (2) above implies that $W_0=0$. If $k$ is an integer
we denote for convenience $W_k=W_{\ov{k}}$, where $\ov{k}$ is
the class of $k$ in $\ZZ/p$. Then
$$\cC^G_H(W)=\cC^G(W_1,\dots,W_{p-1}).$$
Since $H$ acts trivially on $X$, we have $X_H=X\times BH$, so by K\"unneth there is a natural inclusion
$$H^*(X)\otimes H^*(BH)=H^*(X)[\tau]/\{p\tau=0\}\longrightarrow H^*_H(X)$$
where $\tau$ corresponds to $\tau_1\in H^2(BH)$. Let $r_{\chi}$
denote the (complex) rank of $W_{\chi}$. For any $\chi\in\ZZ/p$
let $\CC_{\chi}=X\times\CC$ denote the trivial line bundle over
$X$ with the action of $H$ given by $h\cdot
(x,z)=(x,\rho_{\chi}(h)z)$. By Lemma
\ref{lemma:roots-unity-in-S} the first $H$-equivariant Chern
class of $\CC_{\chi}$ is
$$c_1^H(\CC_{\chi})=\chi\tau.$$
Let $W_{\chi,0}=W_{\chi}\otimes \CC_{-\chi}$. Then $H$ acts
trivially on $W_{\chi,0}$, so the $H$-equivariant Chern class
$c_j^H(W_{\chi,0})$ can be identified with the usual (non
equivariant) Chern class $c_j(W_{\chi,0})$ (which of course
coincides with $c_j(W_{\chi}))$ via the composition of maps
$$H^*(X)\ni\alpha\mapsto \alpha\otimes 1\in
H_H^*(X)\otimes H^0(BH)\subseteq H_H^*(X)\otimes H^*(BH)\to
H^*_H(X).$$ Indeed, the Borel construction applied to
$W_{\chi,0}$ gives a vector bundle on $X_H$ which can be
identified with the pullback of $W_{\chi}$ via the projection
map $X_H\to X$. We may thus write, somewhat abusively,
$c_j^H(W_{\chi,0})=c_j(W_{\chi})$, where here (and below) we
implicitly regard $c_j(W_{\chi})$ as a class in $H_H^*(X)$.
Since $W_{\chi}=W_{\chi,0}\otimes\CC_{\chi}$, we have
$$c_j^H(W_{\chi})=c_j(W_{\chi})+\sum_{k=1}^jc_{j-k}(W_{\chi})(\chi\tau)^k\left( r_{\chi}-j+k\atop k\right).$$
Hence $c_j^H(W_{\chi})\in \cC(W_{\chi})[\tau]$. These formulas
also imply, using induction on $j$, that
$c_j(W_{\chi})\in\cC^H(W_{\chi})[\tau]$. Hence we have:
\begin{equation}
\label{eq:same-rings}
\cC(W_1,\dots,W_{p-1})[\tau]=\cC^H(W_1,\dots,W_{p-1})[\tau]
\end{equation}
Let $r=\sum r_{\chi}$ be the rank of $W$. We have
$$\e^H(W)=\prod_{\chi\in(\ZZ/p)^{\times}}\e^H(W_{\chi})=
\prod_{\chi\in(\ZZ/p)^{\times}}c_{r_{\chi}}^H(W_{\chi})=\prod_{\chi\in(\ZZ/p)^{\times}}
\left(\sum_{j=0}^{r_{\chi}}(\chi\tau)^jc_{r_{\chi}-j}(W_{\chi})\right),$$
so we may write
$$\e^H(W)=a\tau^r+P$$
where $a\in(\ZZ/p)^{\times}$ and $P$ is a polynomial in $\tau$
and the non-equivariant Chern classes
$\{c_j(W_{\chi})\}_{j\geq 0,\chi\in(\ZZ/p)^{\times}}$, all of
whose monomials have at least one factor of the form
$c_j(W_{\chi})$ with $j\geq 1$. In particular, there exists
some natural number $s$ such that $P^{s+1}=0$ (because $X$ is
compact), so
$$\psi_H=a^{-(s+1)}\left(a^s\tau^{rs}+\sum_{j=1}^s(-1)^ja^{s-j}\tau^{r(s-j)}P^j\right)$$
satisfies $\e^H(W)\psi_H=\tau^R$, where $R=r(s+1)$.
Furthermore,
$$\psi_H\in\cC^H(W_1,\dots,W_{p-1})[\tau]$$
by (\ref{eq:same-rings}), so we may write
$$\psi_H=Q(\tau,c_1^H(W_1),\dots,c_{r_1}^H(W_1),\dots,c_1^H(W_{p-1}),\dots,c_{r_{p-1}}^H(W_{p-1})),$$
for some polynomial $Q\in\ZZ[x_0,x_1,\dots,x_r]$.

Now let
$$\nu:H_G^*(X)\to H_H^*(X)$$
be the restriction map and define
$$\psi_G:=Q(\tau_k,c_1^G(W_1),\dots,c_{r_0}^G(W_1),\dots,c_1^G(W_{p-1}),\dots,c_{r_{p-1}}^G(W_{p-1}))
\in H^*_G(X).$$
Since $\nu(\tau_k)=\tau$, $\nu(\e^G(W))=\e^H(W)$ and $\nu(c_j^G(W_{\chi}))=c_j^H(W_{\chi})$, we have
$$\nu(\e^G(W)\psi_G)=\e^H(W)\psi_H=\tau^R=\nu(\tau_k^R).$$
By Lemma \ref{lemma:nilpotence-2}, there exists some $S$ such that
$$(\e^G(W)\psi_G)^S=\tau_k^{RS}.$$
Writing $U=RS$ and $f=\e^G(W)^{S-1}\psi_G^S$, it follows that
$\e^G(W)f=\tau_k^U$, as desired.
\end{proof}

\section{Localization}
\label{s:localization}

Let $p$ be a prime and $k$ a natural number. A {\it $p^k$-admissible
tuple} is a tuple of the form
$$(X,G,A,V_1,\dots,V_r),$$
where:
\begin{itemize}
\item $X$ is a closed smooth manifold, endowed with an
    almost complex structure $J$,
\item $G$ is a finite $p$-group acting smoothly (but non
    necessarily effectively) on $X$ and preserving $J$,
\item $A$ is an abelian normal
subgroup of $G$ such that $X^A\neq\emptyset$,
\item $V_1,\dots,V_r$ are $G$-equivariant complex vector bundles
over $X$,
\end{itemize}
subject to the following condition. Let $N\to X^A$ be the
normal bundle of the inclusion in $X$, with the complex
structure induced by $J$. Denote by $Z_G(A)$ the centralizer of
$A$ in $G$. Then there exist a $Z_G(A)$-invariant open and
closed submanifold $M\subseteq X^A$ and a cohomology class
$$\alpha\in\cC_A(V_1,\dots,V_r,N)$$
such that
$\la [M],\alpha\ra$ is
not divisible by $p^k$.

The following is the main result of the present section.

\begin{theorem}
\label{thm:admissibility-hereditary} Let
$(X,G,A,V_1,\dots,V_r)$ be a $p^k$-admissible tuple. Suppose
that there exists some $h\in Z_G(A)$ whose class $[h]$ in $G/A$
is $G$-invariant and has order $p^k$. Then there exists some
$0\leq i<k$ such that setting $A'=\la A,h^{p^i}\ra$ the tuple
$(X,G,A',V_1,\dots,V_r)$ is $p^k$-admissible.
\end{theorem}
\begin{proof}
Let $M$ be an open and closed submanifold of $X^A$ such that
$M$ is preserved by the action of $Z_{G}(A)$ and such that
there exists some $\alpha\in\cC_A(V_1,\dots,V_r,N)$ whose pairing
$\la [M],\alpha\ra$ is not divisible by $p^k$. Suppose that
$h\in Z_G(A)$ has the property that $[h]\in G/A$ is
$G$-invariant and $[h]$ has order $p^k$ in $G/A$, so
$h^{p^k}\in A$. Define these subgroups of $G$:
$$\Gamma_0\subset \Gamma_1\subset\dots\subset \Gamma_k\subseteq G,
\qquad \Gamma_j:=\la h^{p^{k-j}}\ra\subseteq G.$$
The action of $\Gamma:=\Gamma_k$ on $X$ preserves $M$. Since
$\Gamma_0\subseteq A$, the action of $\Gamma_0$ on $M$ is
trivial, so the action of $\Gamma$ on $M$ induces an action of
$\Gamma/\Gamma_0\simeq\ZZ/p^k\ZZ$ (which need not be
effective).

To simplify the notation, denote by
$$E_1,\dots, E_u$$
the collection of all vector bundles on $M$ of the form
$V_i^{A,\xi}|_M$ for some $i$ and $\xi\in\wh{A}$, or
$N^{A,\xi}|_M$ for some $\xi\in\wh{A}$, where
$N$ is the normal bundle of the inclusion $X^A\subseteq X$. Then we have
$$\alpha\in\cC(E_1,\dots,E_u).$$

\newcommand{\md}{\operatorname{md}}

For any $x\in M^{\Gamma_j}$ define
$$\md_j(x)=(\dim_{\RR}
T_xM^{\Gamma_0},\dim_{\RR} T_xM^{\Gamma_1},\dots,\dim_{\RR}
T_xM^{\Gamma_j})\in\ZZ^{j+1}$$
("$\md$" stands for multi-dimension), and for any
$\mu\in\ZZ^{j+1}$ let
$$M^{\mu}=\{x\in M^{\Gamma_j}\mid \md_j(x)=\mu\}.$$
Define also for any $i\geq 1$ the vector bundle
$$N_i=\frac{TM^{\Gamma_{i-1}}|_{M^{\Gamma_i}}}{TM^{\Gamma_{i}}}.$$
The bundle $N^i$ inherits from $J$ a structure of complex
vector bundle. The normal bundle $N_{M^{\Gamma_j}\subseteq M}$
of the inclusion $M^{\Gamma_j}\subseteq M$ satisfies
$$N_{M^{\Gamma_j}\subseteq M}\simeq \bigoplus_{i=1}^j N_i|_{M^{\Gamma_j}}.$$
In terms of this notation, we have the following.

\begin{lemma}
\label{lemma:localisation-cyclic} There exists some $1\leq
l\leq k$, some $\mu\in\ZZ^{l+1}$, and some cohomology class
$$\beta\in\cC_{\Gamma_l}(E_1,\dots,E_u,N_1,\dots,N_l)$$
such that $\la[M^{\mu}],\beta\ra$ is not divisible by $p^k$.
In particular $M^{\mu}$, and hence $M^{\Gamma_l}$, is nonempty.
\end{lemma}

The proof of Lemma \ref{lemma:localisation-cyclic} is given in
Subsection \ref{ss:proof-lemma:localisation-cyclic} below. We
now prove how the lemma implies Theorem
\ref{thm:admissibility-hereditary}. Let $l,\mu,\beta$ be the
output of Lemma \ref{lemma:localisation-cyclic}.
We are going to prove that if $A'=\la
A,\Gamma_l\ra=\la A,h^{p^{k-l}}\ra$ then
$(X,G,A',V_1,\dots,V_r)$ is $p^k$-admissible.

Since $A$ is normal in $G$ and the class of $h$ in $G/A$
is $G$-invariant, $A'$ is a normal subgroup of $G$. Since
$M^{\Gamma_l}\neq\emptyset$ and $M\subseteq X^A$, we have
$X^{A'}\neq\emptyset$.
We next observe that $M^{\mu}$ is open and closed in $X^{A'}$.
To prove this it suffices to check that $M^{\mu}$ is open and
closed in $M^{\Gamma_l}$ (since $M$ on its turn is open and
closed in $X^A$), and this follows from the fact that the map
$\md_l:M^{\Gamma_l}\to\ZZ^{l+1}$ is locally constant.

The submanifold $M^{\mu}\subseteq X$ is preserved by the action
of $Z_{G}(A')$. This follows from the next two
observations. First, $M\subset X$ is preserved by
$Z_{G}(A)$ and $Z_{G}(A')\subseteq Z_{G}(A)$,
hence $M^{\Gamma_l}$
is preserved by $Z_{G}(A')$. Second, the fact that $[h]\in
G/A$ is $G$-invariant implies that the map
$\md_l:M^{\Gamma_l}\to\ZZ^{l+1}$ is $Z_{G}(A')$-invariant.

Now let $N$ be the normal bundle of the inclusion
$M^{\mu}\subseteq X$. We claim that
\begin{equation}
\label{eq:beta-dins-cC-A}
\beta\in\cC_{A'}(V_1,\dots,V_r,N).
\end{equation}
Since by Lemma \ref{lemma:localisation-cyclic} we have $\la
[M^{\mu}],\beta\ra\notin p^k\ZZ$, (\ref{eq:beta-dins-cC-A})
implies that $(X,G,A',V_1,\dots,V_r)$ is $p^k$-admissible.
The inclusion (\ref{eq:beta-dins-cC-A}) is a consequence of
$$\cC_{\Gamma_l}(E_1,\dots,E_u,N_1,\dots,N_l)\subseteq \cC_{A'}(V_1,\dots,V_r,N),$$
which follows from the product formula for the Chern class of a
direct sum of vector bundles and these two facts:
\begin{itemize}
\item for any $j$ and $\xi\in\wh{\Gamma_l}$ the vector
    bundle $E_j^{\Gamma_l,\xi}$ is the direct sum of vector
    bundles of the form $V_i^{A',\eta}$, for any $i$ and
    any $\eta\in\wh{A'}$, and of the form $N^{A',\eta}$ for
    any $\eta\in\wh{A'}$; this follows from the definition
    of the bundles $E_1,\dots,E_v$, and the fact that
    $E_j^{\Gamma_l,\xi}\subseteq E_j$ is preserved by the
    action of $A'$ (which on its turn follows from the fact
    that $A$ acts on $E_j$ by homotheties);
\item for any $j\leq l$ and $\xi\in\wh{\Gamma_l}$ the vector
    bundle $N_j^{\Gamma_l,\xi}$ is the direct sum of vector
    bundles of the form $N^{A',\eta}$ for any
    $\eta\in\wh{A'}$; this follows from the fact that
$N_j^{\Gamma_l,\xi}\subseteq N$ is $A'$-invariant (which is a
consequence of the assumption that $[h]\in G/A$ is
$G$-invariant).
\end{itemize}
The proof of Theorem \ref{thm:admissibility-hereditary} is now
complete.
\end{proof}

\subsection{Proof of Lemma \ref{lemma:localisation-cyclic}}
\label{ss:proof-lemma:localisation-cyclic} We first introduce
some notation.

Denote $K_i=\Gamma/\Gamma_i$.

For any group $G$ acting on a space $X$ and every
character $\chi\in\wh{G}$, $\CC_{\chi}$ is the trivial line bundle $\CC_{\chi}=X\times\CC$
endowed with the action of $G$ defined by the formula $\gamma\cdot(x,z)=(\gamma\cdot x,\chi(\gamma)z)$.

Since $\Gamma$ is abelian, for every $j$
the inclusion $\Gamma_j\subset \Gamma$ induces by restriction
a surjective morphism of character groups
$r_j:\wh{\Gamma}\to\wh{\Gamma_j}$.
Choose, for each $j$, a set theoretical section $\psi_j:\wh{\Gamma_j}\to\wh{\Gamma}$ of $r_j$.
We emphasize that we
do not require any compatibility relation between the morphisms $\psi_j$ for different
values of $j$. Define, for every $i,j$ and any $\xi\in\wh{\Gamma_j}$, the following $\Gamma$-bundle
over $M^{\Gamma_j}$:
$$F_i(\Gamma_j,\xi)=E_i^{\Gamma_j,\xi}\otimes\CC_{\psi_j(\xi)^{-1}},$$
where we are using multiplicative notation in $\wh{\Gamma}$.
The effect of tensoring $E^{\Gamma_j,\xi}$ by
$\CC_{\psi_j(\xi)^{-1}}$ is to trivialize the action of
$\Gamma_j$. Hence, the natural action of $\Gamma$ on
$F_i(\Gamma_j,\xi)$ factors through an action of $K_j$, which
lifts the action of $K_j$ on $M^{\Gamma_j}$. Similarly, we
define for every $i\leq j$
\begin{equation}
\label{eq:def-R-i}
R_i(\Gamma_j,\xi)=N_i^{\Gamma_j,\xi}\otimes\CC_{\psi_j(\xi)^{-1}},
\end{equation}
which again can be seen as a $\Gamma$-bundle or a $K_j$-bundle over $M^{\Gamma_j}$.

Let $\rho_j:\wh{\Gamma_j}\to\wh{\Gamma_{j-1}}$ be the restriction map.
Take an arbitrary element $\eta\in\wh{\Gamma_{j}}$ and let $\xi=\rho_j(\eta)$.
Since the characters $\psi_j(\eta)\in\wh{\Gamma_j}$ and
$\psi_{j-1}(\xi)\in\wh{\Gamma_{j-1}}$ extend
$\eta$ and $\xi$ respectively
and the restriction of $\eta$ to $\Gamma_{j-1}$ coincides with
$\xi$, the character $\psi_j(\eta)\psi_{j-1}(\xi)^{-1}$ is trivial
on $\Gamma_{j-1}$. Hence, it induces a character of $K_{j-1}$, which we
denote by $\rho(\eta,\xi)$.
For any $j$ and any $\xi\in\wh{\Gamma_{j-1}}$ we have the following
equalities between $K_{j-1}$-bundles:
\begin{equation}
\label{eq:splitting-F}
F_i(\Gamma_{j-1},\xi)|_{M^{\Gamma_j}}=\bigoplus_{\rho_j(\eta)=\xi}F_i(\Gamma_j,\eta)
\otimes\CC_{\rho(\eta,\xi)},
\end{equation}
and
\begin{equation}
\label{eq:splitting-R}
R_i(\Gamma_{j-1},\xi)|_{M^{\Gamma_j}}=\bigoplus_{\rho_j(\eta)=\xi}R_i(\Gamma_j,\eta)
\otimes\CC_{\rho(\eta,\xi)}.
\end{equation}

After these preliminaries, we proceed to prove Lemma
\ref{lemma:localisation-cyclic}. The main ingredient of the
proof will be an inductive argument, which is the content of
Lemma \ref{lemma:inductive-step} below.

We now prepare the setting to run the inductive process.

Since $\alpha\in\cC(E_1,\dots,E_u)$, we may write $\alpha$ as a
sum of monomials on the Chern classes of the bundles $E_i$, and
the hypothesis of the lemma implies that the pairing of $[M]$
with at least one of these monomials yields an integer that is
not divisible by $p^k$. We then replace $\alpha$ by this
monomial, so we have
$$\alpha=\prod_{l=1}^vc_{j_l}(E_{i_l})$$
and the pairing $\la[M],\alpha\ra$ is not divisible by $p^k$.
Let $\mu_0=\sum_{l=1}^v2j_l$. Since the degree of $\alpha$
is $\mu_0$, we have
$\left\la [M],\alpha\right\ra
=\left\la [M^{\mu_0}],\alpha\right\ra,$
so $\left\la [M^{\mu_0}],\alpha\right\ra$ is not divisible by $p^k$. Using the splitting
$E_i\simeq\bigoplus_{\xi\in\wh{\Gamma_0}}E_i^{\Gamma_0,\xi}$ and the multiplicativity of
the total Chern class with respect to Whitney sum, and expanding, it follows that there exists some
cohomology class
$$\alpha_0=\prod_{l=1}^{v_0}c_{j_{0,l}}\left(E_{i_{0,l}}^{\Gamma_0,\xi_{0,l}}\right)$$
such that
\begin{equation}
\label{eq:primer-no-divisible}
\la [M^{\mu_0}],\alpha_0\ra\qquad\text{is not divisible by $p^k$.}
\end{equation}
Let
$$\gamma_0:=
\prod_lc_{j_{0,l}}^{K_0}\left(F_{i_{0,l}}(\Gamma_0,\xi_{0,l})\right)
\in H^*_{K_0}(M^{\mu_0}).$$
Let $\pi_0:M^{\mu_0}\to \{*\}$ the map to a point. Then
$$(\pi_0)^{K_0}_*(\gamma_0)\in H^0(BK_0)\simeq\ZZ,$$
and it follows from (\ref{eq:primer-no-divisible}) and property
(P4) of \S\ref{s:umkehrungs} applied to the inclusion of groups
$\{1\}\hookrightarrow K_0$ that $\pi^{K_0}_*(\gamma_0)$ is
generic in the sense of Section \ref{s:equivariant-cohomology}.
Now we apply repeatedly the following lemma, beginning at $s=0$
and stopping as soon as we fall in case (I), in which case the
proof of Lemma \ref{lemma:localisation-cyclic} is completed.

\begin{lemma}
\label{lemma:inductive-step}
Let $s\geq 0$. Suppose there is some $\mu=(\mu_0,\dots,\mu_s)\in\ZZ^{s+1}$
and a class
$$\gamma_s=\prod_l c_{j_{s,l}}^{K_s}\left(F_{i_{s,l}}(\Gamma_s,\xi_{s,l})\right)
\prod_l
c_{j_{s,l}'}^{K_s}\left(R_{i_{s,l}'}(\Gamma_s,\xi_{s,l}')\right)
\in H^*_{K_s}(M^{\mu}),$$ where $i_{s,l}'\leq s$ for every $l$
(it should be understood that when $s=0$ no bundle of the type
$R_j$ appears in the formula) such that, letting
$\pi_s:M^{\mu}\to\{*\}$ denote the map to a point, the
cohomology class $(\pi_s)^{K_s}_*(\gamma_s) \in H^*(BK_s)$ is
generic. Then there exists an integer $\mu_{s+1}$ such that,
letting $\nu=(\mu_0,\dots,\mu_s,\mu_{s+1})$, at least one of
the following two properties holds true:
\begin{enumerate}
\item[(I)] there exists some $\beta\in\cC_{\Gamma_{s+1}}(E_1,\dots,E_u,N_1,\dots,N_{s+1})$
such that $\la [M^{\nu}],\beta\ra$ is not divisible by $p^k$, or
\item[(II)] denoting by $\pi_{s+1}:M^{\nu}\to\{*\}$ the map to a point, there exists a
class
$$\gamma_{s+1}=\prod_l c_{j_{s+1,l}}^{K_{s+1}}\left(F_{i_{s+1,l}}(\Gamma_{s+1},\xi_{s+1,l})\right)
\prod_l
c_{j_{s+1,l}'}^{K_{s+1}}\left(R_{i_{s+1,l}'}(\Gamma_{s+1},\xi_{s+1,l}')\right)$$
in $H^*_{K_{s+1}}(M^{\nu})$ such that
$(\pi_{s+1})^{K_{s+1}}_*(\gamma_{s+1})$ is generic and
$i_{s+1,l}'\leq s+1$ for every $l$.
\end{enumerate}
\end{lemma}
\begin{proof}
Since $(\pi_s)^{K_s}_*(\gamma_s)\in H^*(BK_s)$ is generic,
Lemma \ref{lemma:free-action-roots-unity} implies that the
action of $K_s$ on $M^{\mu}$ is not free, so there is some
subgroup $K'\subseteq K_s$ such that $K'\neq\{1\}$ and
$(M^{\mu})^{K'}\neq\emptyset$. Since any subgroup $K'\subseteq
K_s$ different from $\{1\}$ contains the projection of
$\Gamma_{s+1}$ to $K_s=\Gamma/\Gamma_s$, it follows that
$(M^{\mu})^{\Gamma_{s+1}}\neq\emptyset$. For any $n\in
\ZZ_{\geq 0}$ let
$$M^{\mu}_n=\bigcup_{X\subseteq M^{\mu}\text{ connected component}\atop
\dim_{\RR} X=n,\,X\cap (M^{\mu})^{\Gamma_{s+1}}\neq\emptyset}X,$$
and let
$$M^{\mu}_{-1}=M^{\mu}\setminus\bigcup_{n\geq 0}M^{\mu}_n.$$
Note that for each $k\geq-1$ the subset $M^{\mu}_k\subseteq
M^{\mu}$ is $K_s$-invariant. Since the action of $K_s$ on
$M^{\mu}_{-1}$ is free,
$(\pi_s)^{K_s}_*(\gamma_s|_{M^{\mu}_{-1}})$ is not generic.
Since
$$\gamma_s|_{M^{\mu}}=\sum_{k\geq -1}\gamma_s|_{M^{\mu}_k}$$
(where $\gamma_s|_{M^{\mu}_k}\in H^*_{K_s}(M^{\mu})$ is meant
to be zero on $M^{\mu}\setminus M^{\mu}_k$), it follows that
there is some $\mu_{s+1}\in\ZZ_{\geq 0}$ such that
$(\pi_s)^{K_s}_*(\gamma_s|_{M^{\mu}_{\mu_{s+1}}})$ is generic.
Let $\nu=(\mu_0,\dots,\mu_s,\mu_{s+1})$, let
$\iota_s:M^{\nu}\to M^{\mu}$ be the inclusion and let
$\pi_{s+1}:M^{\nu}\to\{*\}$ be the projection to a point. Since
$M^{\nu}=M^{\mu}_{\mu_{s+1}}\cap M^{\Gamma_{s+1}}$, the action
of $K_s$ on
$$(M^{\mu}_{\mu_{s+1}})^*=M^{\mu}_{\mu_{s+1}}\setminus M^{\nu}$$
is free. Let $\theta:K_s\to\CC^*$ be an injective morphism and let
$$\tau=c_1^{K_s}(EK_s\times_{\theta}\CC)\in H^2_{K_s}(BK_s).$$
We also denote by $\tau$ the pullback of this cohomology class to $H^*_{K_s}(M^{\mu})$ via
the morphism induced by the projection $\pi_s$. Since the action of $K_s$ on
$(M^{\mu}_{\mu_{s+1}})^*$ is free, $H^v_{K_s}((M^{\mu}_{\mu_{s+1}})^*)$ vanishes for
big enough $v$. Hence there exists some $r$ such that
$$\tau^r\gamma_s|_{(M^{\mu}_{\mu_{s+1}})^*}=0.$$
The long exact sequence in equivariant cohomology for the pair $(M^{\mu}_{\mu_{s+1}},(M^{\mu}_{\mu_{s+1}})^*)$ implies that
$\tau^r\gamma_s$ belongs to the image of the map
$$H^*_{K_s}(M^{\mu}_{\mu_{s+1}},(M^{\mu}_{\mu_{s+1}})^*)\to H^*_{K_s}((M^{\mu}_{\mu_{s+1}})^*).$$
By (P3) in \S\ref{s:umkehrungs} there exists some $\lambda_s\in
H^*_{K_s}(M^{\nu})$ such that
$\tau^r\gamma_s=(\iota_s)^{K_s}_*(\lambda_s)$ so
$$\iota_s^*(\tau^r\gamma_s)=\lambda_s\e^{K_s}(N_{s+1}|_{M^{\nu}}).$$
According to Theorem \ref{thm:inverting-euler-class} there exists
$$f\in\cC^{K_s}(\{N_{s+1}^{\Gamma_{s+1},\chi}\}_{\chi\in\wh{\Gamma_{s+1}}})[\tau]\subseteq H^*_{K_s}(M^{\nu})$$
and some nonnegative integer $U$ such that
$$\lambda_s\e^{K_s}(N_{s+1}|_{M^{\nu}})f=\tau^U\lambda_s.$$
Then
\begin{align*}
(\pi_{s+1})^{K_s}_*(\iota_s^*(\tau^r\gamma_s)f) &= (\pi_{s+1})^{K_s}_*(\tau^U\lambda_s) \\
&=\tau^U(\pi_{s+1})^{K_s}_*(\lambda_s) \qquad\text{by the product formula, (P2) in \S\ref{s:umkehrungs}} \\
&=\tau^U(\pi_{s}\circ\iota_s)^{K_s}_*(\lambda_s) \qquad \text{since $\pi_{s+1}=\pi_s\circ\iota_s$}\\
&=\tau^U(\pi_{s})^{K_s}_*(\tau^r\gamma_s) \\
&=\tau^{U+r}(\pi_{s})^{K_s}_*(\gamma_s),
\end{align*}
so $(\pi_{s+1})^{K_s}_*(\iota_s^*(\tau^r\gamma_s)f)$ is
generic. By (\ref{eq:def-R-i}),
$$\cC^{K_s}(\{N_{s+1}^{\Gamma_{s+1},\chi}\}_{\chi\in\wh{\Gamma_{s+1}}})[\tau]
=
\cC^{K_s}(\{R_{s+1}(\Gamma_{s+1},\chi)\}_{\chi\in\wh{\Gamma_{s+1}}})[\tau],$$
so we may write $$f=\tau^{e_1}f_1+\tau^{e_2}f_2+\dots,$$
where each $f_i$ is a monomial on the $K_s$-equivariant
Chern classes of the bundles
$$\{R_{s+1}(\Gamma_{s+1},\chi)\}_{\chi\in\wh{\Gamma_{s+1}}}.$$
Since $(\pi_{s+1})^{K_s}_*(\iota_s^*(\tau^r\gamma_s)f)$ is
generic, by the product formula and the basic properties of
generic elements in $H^*(BK_s)$, there exists some $w$ such
that $(\pi_{s+1})^{K_s}_*(\iota_s^*(\gamma_s)f_w)$ is generic.
We may write
$$\iota_s^*(\gamma_s)f_w=
\prod_l c_{j_{s,l}}^{K_s}\left(F_{i_{s,l}}(\Gamma_s,\xi_{s,l})\right)
\prod_l c_{j_{s,l}'}^{K_s}\left(R_{i_{s,l}'}(\Gamma_s,\xi_{s,l}')\right)
\prod_l c_{j_{s,l}''}^{K_s}\left(R_{s+1}(\Gamma_s,\xi_{s,l}'')\right),$$
where the factors in the first two products are implicitly restricted to $M^{\nu}$.

By formulas (\ref{eq:splitting-F}) and (\ref{eq:splitting-R}) we have
$$c_j^{K_s}(F_i(\Gamma_s,\xi))|_{M^{\nu}}\in\cC^{K_s}(\{F_i(\Gamma_{s+1},\chi)\}_{\chi\in\wh{\Gamma_{s+1}}})[\tau]$$
for every $j,i,\xi$ and similarly
$$c_j^{K_s}(R_i(\Gamma_s,\xi))|_{M^{\nu}}\in\cC^{K_s}(\{R_i(\Gamma_{s+1},\chi)\}_{\chi\in\wh{\Gamma_{s+1}}})[\tau]$$
for every $j,i\leq s+1,\xi$. Consider the following collection of vector bundles
over $M^{\nu}$:
$$\bB=\{F_i(\Gamma_{s+1},\chi)\}_{i,\chi\in\wh{\Gamma_{s+1}}}\cup
\{R_i(\Gamma_{s+1},\chi)\}_{i\leq s+1,\chi\in\wh{\Gamma_{s+1}}}.$$
We have
$$\iota_s^*(\gamma_s)f_w\in
\cC^{K_s}(\bB)[\tau],$$
so we may write
$$\iota_s^*(\gamma_s)f_w=\tau^{d_1}g_1+\tau^{d_2}g_2+\dots$$
where each $g_i$ is a monomial on the $K_s$-equivariant Chern
classes of the bundles in $\bB$. For at least one $z$, the class
$(\pi_{s+1})^{K_s}_*(g_z)$ is generic. Write
$$g_z=\prod_l c_{j_{s+1,l}}^{K_{s}}\left(F_{i_{s+1,l}}(\Gamma_{s+1},\xi_{s+1,l})\right)
\prod_l c_{j_{s+1,l}'}^{K_{s}}\left(R_{i_{s+1,l}'}(\Gamma_{s+1},\xi_{s+1,l}')\right)
\in H^*_{K_{s}}(M^{\nu}).$$

Now we distinguish two cases.

Suppose first that $\deg g_z$
is equal to $\dim M^{\nu}$. In that case we define
$$\beta:=
\prod_l
c_{j_{s+1,l}}\left(E_{i_{s+1,l}}^{\Gamma_{s+1},\xi_{s+1,l}}\right)
\prod_l
c_{j_{s+1,l}'}\left(N_{i_{s+1,l}'}^{\Gamma_{s+1},\xi_{s+1,l}'}\right)$$
and by property (P4) of \S\ref{s:umkehrungs} applied to the
inclusion of groups $\{1\}\hookrightarrow K_s$ we may identify
$$(\pi_{s+1})^{K_s}_*(g_z)\in H^0(BK_s)\simeq\ZZ$$
with the paring $\la [M^{\nu}],\beta\ra$. Since the former is generic,
it follows that $\la [M^{\nu}],\beta\ra$ is not divisible by $p^{k-s}=|K_s|$, so
a fortiori it is not divisible by $p^k$. Hence we fall in case (I) of the statement of
the lemma.

Now suppose that $\deg g_z>\dim M^{\nu}$ and define
$$\gamma_{s+1}=\prod_l c_{j_{s+1,l}}^{K_{s+1}}\left(F_{i_{s+1,l}}(\Gamma_{s+1},\xi_{s+1,l})\right)
\prod_l
c_{j_{s+1,l}'}^{K_{s+1}}\left(R_{i_{s+1,l}'}(\Gamma_{s+1},\xi_{s+1,l}')\right)
\in H^*_{K_{s+1}}(M^{\nu}).$$ Let $\rho:K_s\to K_{s+1}$ be the
natural projection map, and let
$\rho^*_M:H^*_{K_{s+1}}(M^{\nu})\to H^*_{K_{s}}(M^{\nu})$,
$\rho^*:H^*(BK_{s+1})\to H^*(BK_{s})$ be the maps induced by
$\rho$ (for the first map, recall that the action of $K_s$ on
$M^{\nu}$ factors through $\rho$). Then we have
$\rho_M^*\gamma_{s+1}=g_z$. By (P4) in \S\ref{s:umkehrungs},
$$\rho^*((\pi_{s+1})^{K_{s+1}}_*(\gamma_{s+1}))=(\pi_{s+1})^{K_{s}}_*(g_z).$$
Since $(\pi_{s+1})^{K_{s}}_*(g_z)\in H^*(BK_s)$ is generic, in
particular it is nonzero, which implies that
$(\pi_{s+1})^{K_{s+1}}_*(\gamma_{s+1})$ is nonzero. On the
other hand,
$$(\pi_{s+1})^{K_{s+1}}_*(\gamma_{s+1})\in H^{>0}(BK_{s+1})$$
because $\deg\gamma_{s+1}=\deg g_z>M^{\nu}$. Since a class in
$H^{>0}(BK_{s+1})$ is generic if and only if it is nonzero, it
follows that $(\pi_{s+1})^{K_{s+1}}_*(\gamma_{s+1})$ is
generic. So we fall in case (II) of the lemma, and the proof is
complete.
\end{proof}

\section{Finite $p$-groups}
\label{s:finite-p-groups}

This section contains several results on finite $p$-groups that
will be used in the proof of Theorem
\ref{thm:main-fixed-point}. Before going to the results, we
introduce some notation.

For any finite group $G$ we define
$$\alpha(G):=\min\{[G:A]\mid A\subseteq G\text{ abelian
subgroup}\}.$$ We denote
by $e(G)$ the exponent of $G$ (the least
common multiple of the orders of its elements).
For a finite abelian group $A$ we denote by $\rk(A)$ the
minimal number of elements needed to generate $A$. The rank of
a finite group $G$, denoted $\rk(G)$, is by definition the
maximum of $\rk(A)$ as $A$ runs over the set of all abelian
subgroups of $G$. If $p$ is a prime number and $G$ is a finite
$p$-group then $\rk(G)$ coincides with the biggest
$r$ such that $G$ contains a subgroup isomorphic to
$(\ZZ/p)^r$.

Throughout this section we use additive notation for finite
abelian groups, and multiplicative notation for not necessarily
abelian groups (such as automorphism groups of
abelian groups).

\subsection{Abelian subgroups of finite $p$-groups of bounded rank}

The main result of this subsection is the following theorem.

\begin{theorem}
\label{thm:cyclic-by-abelian} Let $a,r$ be natural numbers and
let $p$ be a prime. Suppose that $G$ is a finite $p$-group
satisfying $\rk(G)\leq r$ and $\alpha(G)>a^{18 r^5}$. Then
there exist subgroups $B\unlhd \Gamma\subseteq G$ such that $B$
is abelian, $\Gamma/B$ is cyclic, and $\alpha(\Gamma)>a$.
\end{theorem}

In the proof of theorem we will need the following four lemmas,
the first three of which are well known. The proof of Theorem
\ref{thm:cyclic-by-abelian} is given after Lemma
\ref{lemma:alpha-nilpotent-2}.

In what follows, $p$ denotes an arbitrary prime.
The first lemma is well known (see e.g. \cite[\S 5.2.3]{Rob}) and will be used
several times.

\begin{lemma}
\label{lemma:maximal-normal-abelian} Let $B$ be a maximal abelian normal subgroup
of a finite $p$-group $G$. Then the action by conjugation on $B$ identifies $G/B$ with a
subgroup of $\Aut(B)$.
\end{lemma}

The following lemma gives an upper bound on the size of a
finite $p$-group in terms of the rank and the exponent. Such
bounds have been much studied in the literature. We give a
simple argument for completeness, but much better bounds can be
obtained using more elaborated arguments (see for example
\cite{S}).

\begin{lemma}
\label{lemma:rank-exponent} If $G$ is a finite $p$-group
satisfying $\rk(G)=r$ and $e(G)=p^d$ then $|G|\leq p^{2dr^2}$.
\end{lemma}
\begin{proof}
Let $B\subseteq G$ be a maximal abelian normal subgroup. Since
$\rk(B)\leq r$ and $e(B)\leq p^d$, we have $|B|\leq p^{rd}$. A
theorem of Hall (see e.g. \cite[\S 5.3.3]{Rob})
implies that the order of $\Aut(B)$ divides
$p^{r^2(d-1)}|\GL_r(\FF_p)|=p^{r^2(d-1)}(p^r-1)(p^r-p)\dots
(p^r-p^{r-1})$. Hence, any Sylow $p$-subgroup $S\subseteq
\Aut(B)$ satisfies
$$|S|\leq p^{r^2(d-1)}p^{1+2+\dots+(r-1)}=
p^{r^2(d-1)}p^{\frac{r(r-1)}{2}}=p^{\frac{r(2rd-r-1)}{2}}.$$ By
Lemma \ref{lemma:maximal-normal-abelian} we can identify $G/B$
with a $p$-subgroup of $\Aut(B)$. Consequently,
$$|G|=|B|\cdot|G/B|\leq p^{rd}p^{\frac{r(2rd-r-1)}{2}}\leq p^{2dr^2}.$$
\end{proof}

The following is Gorchakov--Hall--Merzlyakov--Roseblade's
lemma (see \cite[Lemma 5]{Ro}).

\begin{lemma}
\label{lemma:GHMR} Let $B$ be a finite abelian $p$-group and
let $r=\rk(B)$. We have $$\rk(\Aut(B))\leq r(5r-1)/2.$$
\end{lemma}

The proof of the following lemma was kindly supplied to us by A. Jaikin.

\begin{lemma}
\label{lemma:alpha-nilpotent-2} Suppose that $B\unlhd \Gamma$
are finite $p$-groups, with $B$ abelian, and that the
morphism $\Gamma/B\to\Aut(B)$ given
by conjugation is injective. Let $r=\rk(B)$. Then
$$\alpha(\Gamma)\geq |\Gamma/B|^{\frac{1}{r}}.$$
\end{lemma}
\begin{proof}
Let $A$ be an arbitrary abelian subgroup of $\Gamma$. Since $B$ has rank
$r$, we can choose characters $\lambda_1,\dots,\lambda_r\in\wh{B}=\Hom(B,\CC^*)$ such that
$\bigcap_{i=1}^r\Ker\lambda_i=\{1\}$. Let $C=B/A\cap B$.
To each $\ov{a}\in A/A\cap B$ we associate the $r$-tuple
$(a_1,\dots,a_r)\in (\wh{C})^r$ by means of $a_i(\ov{b})=\lambda_i([a,b])$.
Since the centraliser of $B$ in $A$ is $A\cap B$, different elements of $A/A\cap B$ correspond to
different $r$-tuples. Thus
$$|A/A\cap B|\leq |(\wh{C})^r|=|C|^r.$$
Let $G$ be the subgroup of $\Gamma$ generated by $A$ and $B$. Since $B$ is normal in $G$, $G=AB$. Thus
$[G:A]=[B:A\cap B]$ and $[G:B]=[A:A\cap B]$. Hence we obtain
\begin{multline*}
[\Gamma:A]=[\Gamma:G][G:A]=[\Gamma:G][B:A\cap B]\geq \\
\geq [\Gamma:G][A/A\cap B]^{1/r}\geq ([\Gamma:G][G:B])^{1/r}=[\Gamma:B]^{1/r},
\end{multline*}
which is what we wanted to prove.
\end{proof}

We are now ready to prove Theorem \ref{thm:cyclic-by-abelian}.
Let $a,r$ be natural numbers, and let $p$ be a prime. Let $G$
be a finite $p$-group satisfying $\rk(G)\leq r$ and
$\alpha(G)>a^{18 r^5}$. Let $B\unlhd G$ be a maximal normal
abelian subgroup. By Lemma \ref{lemma:maximal-normal-abelian},
the action of $G$ on $B$ by conjugation induces an injective
morphism $c:G/B\to\Aut(B)$. We have
$$|G/B|\geq\alpha(G)>a^{18 r^5}.$$ By Lemma
\ref{lemma:GHMR}, $\rk(G/B)\leq 3r^2$. Lemma
\ref{lemma:rank-exponent} implies that
$$e(G/B)\geq |G/B|^{\frac{1}{2\rk(G/B)^2}}>(a^{18r^5})^{\frac{1}{18r^4}}=a^{r}.$$
Since $G/B$ is a $p$-group,
the exponent of $G/B$ coincides with the maximum of the orders
of its elements. So there exists some element $\gamma\in
G/B$ whose order satisfies $\ord(\gamma)>a^{r}$. Let
$Z=\la\gamma\ra\subseteq G/B$ and let $\pi:G\to G/B$ be
the quotient map. Define $\Gamma:=\pi^{-1}(Z)$. Then the groups
$B\unlhd \Gamma$ satisfy the hypothesis of Lemma
\ref{lemma:alpha-nilpotent-2}, so
$$\alpha(\Gamma)\geq |\Gamma/B|^{\frac{1}{r}}=|Z|^{\frac{1}{r}}=
\ord(\gamma)^{\frac{1}{r}}>a.$$ This completes
the proof of the theorem.

\subsection{Some technical results}
\label{s:more-on-p-groups}
We begin recalling a well known and elementary lemma.

\begin{lemma}
\label{lemma:element-h-p} Let $B$ be a finite abelian
$p$-group, let $\Phi\subseteq \Aut(B)$ be a $p$-subgroup, and let
$A\subset B$ be a $\Phi$-invariant proper subgroup. There
exists some $h\in B$ whose class in $B/A$ is $\Phi$-invariant
and has order $p$.
\end{lemma}
\begin{proof}
Consider the natural action of $\Phi$ on $B/A$. The element
$[0]\in B/A$ is fixed by $\Phi$, so its orbit has one element.
Since $|B/A|$ is divisible by $p$, there must be some
nontrivial $[g]\in B/A$ whose orbit consists also of a unique
element. Let $p^e$ be the order of $[g]$. Then $h:=p^{e-1}g$
has the desired properties.
\end{proof}

Our aim in this section is to generalize the previous lemma so
as to obtain elements whose class in $B/A$ is $\Phi$-invariant
and has order $p^k$, where $k$ is an arbitrary natural number.
Actually, we will generalize a weaker form of the lemma, in
which $\Phi$ is required to act trivially on $B/pB$.
To state the result we introduce some notation.

Suppose that $B$ is a finite abelian
$p$-group, and $k$ is a natural number. Define
$$\Aut_{(k)}(B)=\{\phi\in\Aut(B)\mid \phi(g)\in g+p^kB\text{
for all $g\in B$}\}.$$ The binomial formula implies that for
any $\phi\in \Aut_{(k)}(B)$ we have $\phi^p\in \Aut_{(k+1)}(B)$
(write $\phi=\Id+\psi$, where $\psi(g)\in p^kB$ for every $g$;
then $\phi^p=\Id+\sum_{k\geq 1}\left(p\atop k\right)\psi^k$).
Since $\Aut_{(k)}(B)=\{1\}$ for big enough $k$, it follows that
$\Aut_{(k)}(B)$ is a $p$-group.

We are now ready for our generalization of Lemma \ref{lemma:element-h-p}.

\begin{theorem}
\label{thm:element-h-p-k} Let $B$ be a finite abelian
$p$-group, let $r=\rk(B)$, and let $k$ be a natural number.
Let $A\subseteq B$ be an $\Aut_{(k)}(B)$-invariant
subgroup. There exists some $h\in B$ whose class in $B/A$ is
$\Aut_{(k)}(B)$-invariant and has order $p^k$, unless $|B/A|\leq
p^{r(k-1)}$.
\end{theorem}
We remark that the condition that $hA\in B/A$ is
$\Aut_{(k)}(B)$-invariant implies that $\la A,nh\ra\subseteq B$ is
$\Aut_{(k)}(B)$-invariant for every $n\in\ZZ$.
\begin{proof}
Let $\{b_1,\dots,b_r\}$ be a generating set of $B$ and define,
for every $i$, $$c_i=\min\{j\in\ZZ_{\geq 0}\mid p^jb_i\in
A\}.$$ Let $c=\max\{c_1,\dots,c_r\}.$

Suppose that $c\geq k$. We prove that in this case there exists
some $h\in B$ with the desired properties. Suppose that $c=c_i$
and define
$$h=p^{c_i-k}b_i=p^{c-k}b_i.$$
Let us check that $[h]\in B/A$ is $\Aut_{(k)}(B)$-invariant.
Equivalently, we have to prove that $\phi(h)\in h+A$ for all
$\phi\in\Aut_{(k)}(B)$. Let $\phi\in\Aut_{(k)}(B)$. We have
$$\phi(b_i)=b_i+p^k(m_1b_1+\dots+m_rb_r)$$ for some
$m_1,\dots,m_r\in\ZZ$, so
$$\phi(h)=\phi(p^{c-k}b_i)=p^{c-k}b_i+p^c(m_1b_1+\dots+m_rb_r)\in h+A$$
because $p^cb_j\in A$ for every $j$. The order of $[h]\in B/A$
is a power of $p$, so it follows from the definitions of $c$
and $h$ that $[h]$ has order $p^k$.

Hence if $h$ does not exist then $c\leq k-1$. But in that case the map
$$(\ZZ/p^{k-1})^r\ni (l_1,\dots,l_r)\mapsto \sum l_ib_i\in
B/A$$ is well defined and surjective, so $|B/A|\leq p^{r(k-1)}$.
\end{proof}

We will later use the following lemma.

\begin{lemma}
\label{lemma:index-Aut(k)} Let $B$ be a finite abelian $p$-group and let $k$ be a
natural number. Then $\Aut_{(k)}(B)$ is a normal
$p$-subgroup of $\Aut(B)$, so it is contained in any Sylow
$p$-subgroup of $\Aut(B)$. Let $r=\rk(B)$ and let
$\Aut(B)_p\subseteq\Aut(B)$ be a Sylow $p$-subgroup. We have
$$[\Aut(B)_p:\Aut_{(k)}(B)]\leq p^{kr^2}.$$
\end{lemma}
\begin{proof}
The subgroup $p^kB\subseteq B$ is characteristic, so
$\Aut_{(k)}(B)$ is a normal subgroup of $\Aut(B)$. Moreover,
the quotient $\Aut(B)/\Aut_{(k)}(B)$ can be identified with a
subgroup of $\Aut(B/p^kB)$. By the arguments in the proof of Lemma
\ref{lemma:rank-exponent}, if $S\subseteq
\Aut(B/p^kB)$ is a $p$-subgroup then $|S|\leq p^{kr^2}$, so
$[\Aut(B)_p:\Aut_{(k)}(B)]=|\Aut(B)_p/\Aut_{(k)}(B)|\leq p^{kr^2}.$
\end{proof}

\section{Vector bundles with nonzero Euler number}
\label{s:proof-thm:main-fixed-point}

The main purpose of this section is to prove Theorem \ref{thm:main-fixed-point}.

\subsection{Preliminaries}

\begin{lemma}
\label{lemma:cohom-no-augmenta}
Let $X$ be a manifold, let $p$ be a prime, and
let $G$ be a finite $p$-group acting continuously on $X$.
We have
$$\sum_j b_j(X^G;\FF_p)\leq \sum_j b_j(X;\FF_p).$$
In particular, the number of connected components of $X^G$ is
at most $\sum_j b_j(X;\FF_p)$.
\end{lemma}
\begin{proof}
If $|G|=p$ then the statement follows from \cite[Theorem III.4.3]{Bo}.
For general $G$ use ascending induction on $|G|$. In the induction step,
choose a central subgroup $G_0\subset G$ of order $p$ and apply the inductive
hypothesis to the action of $G/G_0$ on $X^{G_0}$.
\end{proof}

The following classical theorem of Camille Jordan was already
mentioned in the introduction (see \cite{J}, and \cite{CR,M1}
for modern proofs).
\begin{theorem}[Jordan]
\label{thm:Jordan-classic} For any natural number $n$ there is some constant $\Jor_n$
such that any finite subgroup $G\subset\GL(n,\RR)$ has an abelian subgroup $A$
satisfying $[G:A]\leq \Jor_n.$
\end{theorem}

\begin{corollary}
\label{cor:Jordan-classic} Let $X$ be an $m$-dimensional
connected smooth manifold and let $G$ be a finite group acting
smoothly and effectively on $X$. If $X^G\neq\emptyset$ then
there exists an abelian subgroup $A\subseteq G$ such that
$[G:A]\leq\Jor_m$. If furthermore $m=2n$ and there is a
$G$-invariant almost complex structure on $X$ then $A$ can be
generated by $n$ elements.
\end{corollary}
\begin{proof}
Take a $G$-invariant Riemannian metric $\nu$ on $X$. For any
subgroup $H\subseteq G$ and any $x\in X^H$ the exponential map
$\exp_x^{\nu}:T_xX\to X$ is $H$-equivariant and induces a
diffeomorphism between a neighborhood of $0$ in $T_xX$ and a
neighborhood of $x$ in $X$. Hence, $X^H$ is a smooth
submanifold of $X$, and $X^H\neq X$ unless $H=\{1\}$, because
$G$ acts effectively. Since $X$ is connected, we deduce that
$X^H$ has empty interior unless $H=\{1\}$. So for any $x\in
X^G$ the derivative of the action of $G$ at $x$ induces an {\it
injective} morphism $G\hookrightarrow\GL(T_x)$: indeed, if the
action of some nontrivial element $g\in G$ on $T_xX$ were
trivial, then by the $G$-equivariance of $\exp_x^{\nu}:T_xX\to
X$ there would exist some neighbourhood of $x$ on which $g$
would act trivially and $X^{\la g\ra}$ would have nonempty
interior. Since $\GL(T_xX)\simeq\GL(m,\RR)$, the first part of
the corollary follows from Theorem \ref{thm:Jordan-classic}.
Finally, if there is a $G$-invariant almost complex structure
on $X$ then $A$ can be identified with an abelian subgroup of
$U(n)$, and this immediately implies that $A$ can be generated
by $n$ elements.
\end{proof}

We will use the following theorem of Mann and Su, see
\cite[Theorem 2.5]{MS}.

\begin{theorem}
\label{thm:MS} For any compact manifold $X$ there
exists some integer $r\in\ZZ$ depending only on $H^*(X)$ with
the property that for any prime $p$ and any elementary
$p$-group $(\ZZ/p)^s$ admitting an effective action on $X$ we
have $s\leq r$.
\end{theorem}

\subsection{Proof of Theorem \ref{thm:main-fixed-point}}
\label{ss:proof-of-thm-fixed-pt}  Let $(X,J)$ be an
almost complex, compact, connected and
orientable manifold, let $E\to X$ be a complex vector bundle
satisfying $\rk_{\RR}(E)=\dim_{\RR}X$ and assume that
$\la[X],\e(E)\ra\neq 0$.

Let $E^*\subset E$ be the complementary of the zero section,
and let $S(E)$ be the sphere bundle of $E$, defined as the
quotient of $E^*$ by the action of $\RR_{>0}$ given by
multiplication. Any finite group acting effectively on $E$ by
vector bundle automorphisms acts also effectively on $S(E)$.
Let $r$ be the number given by Mann and Su's Theorem
\ref{thm:MS} applied to $S(E)$. The cohomology of
$H^*(S(E))$ can be computed in terms of $H^*(X)$ and the Euler
class $\e(E)$ by Gysin, and $\e(E)$ is uniquely
determined by $\la[X],\e(E)\ra$, so $r$ can be chosen to depend
only on $H^*(X)$ and $\la[X],\e(E)\ra$.

Fix some prime $p$. Define the natural number $k$ by the
condition that $p^k$ is the smallest power of $p$ not dividing
$\la[X],\e(E)\ra$. Consider a finite $p$-group $G$ acting
smoothly and effectively on $E$ by vector bundle automorphisms,
and lifting an action on $X$ which preserves $J$.
It follows from the definition of $r$ in
the previous paragraph that
$$\rk(G)\leq r.$$
We will prove the theorem in two steps.

\noindent{\bf Step 1.} Suppose that $G$ sits in an exact
sequence
\begin{equation}
\label{eq:1-B-G-H-1}
1\longrightarrow
B \longrightarrow G \stackrel{\pi}{\longrightarrow}
H\longrightarrow 1,
\end{equation}
where $B$ is an abelian normal subgroup
of $G$ and $H$ is cyclic. The action of $G$ on $B$ by
conjugation induces a morphism $c:H\to\Aut(B)$. Define
$$H_0:=c^{-1}(\Aut_{(k)}(B))$$
if $k>1$ (see Section \ref{s:more-on-p-groups})
and $H_0:=H$ if $k=1$.
By Lemma \ref{lemma:index-Aut(k)},
\begin{equation}
\label{eq:H:H_0}
[H:H_0]\leq p^{kr^2}\quad\text{if $k>1$,}\qquad\qquad
[H:H_0]=1\quad\text{if $k=1$}.
\end{equation}
Let $G_0=\pi^{-1}(H_0)$. The tuple $(X,G_0,\{1\},E)$ is
$p^k$-admissible in the sense of Section \ref{s:localization}.

\begin{lemma}
\label{lemma:grup-A-0}
There exists an
abelian subgroup $A_0\subseteq B$, such that:
\begin{enumerate}
\item $A_0$ is normal in $G_0$;
\item $X^{A_0}\neq\emptyset$, and there exists some
    cohomology class $\alpha\in \cC_{A_0}(E,N)$, where $N$
    is the normal bundle of the inclusion $X^{A_0}\subseteq
    X$, and an open and closed submanifold
    $$M\subseteq X^{A_0}$$
    such that $\la [M],\alpha\ra$ is not
    divisible by $p^k$;
\item $[B:A_0]\leq p^{r(k-1)}$.
\end{enumerate}
\end{lemma}
\begin{proof}
Let $\aA_0$ be the set of subgroups of $B$ satisfying (1) and (2).
We have $\{1\}\in\aA_0$. We claim that any maximal element of $\aA_0$ satisfies
(3). To prove this we consider two cases. If $k>1$ then this follows from
Theorems \ref{thm:element-h-p-k} and \ref{thm:admissibility-hereditary}.
If $k=1$ it follows from Lemma
\ref{lemma:element-h-p} (with $\Phi$ given by the adjoint action of $G_0$)
and Theorem \ref{thm:admissibility-hereditary}.
\end{proof}

Let $A_0\subseteq B$ be one of the subgroups given by Lemma \ref{lemma:grup-A-0}.
Since $A_0$ is normal in $G_0$, the action of $G_0$
on $X$ preserves $X^{A_0}$. Let $G_1\subseteq G_0$ be the
subgroup consisting of those elements that preserve each connected component of $X^{A_0}$
(in particular, $G_1$ preserves $M$). By Lemma
\ref{lemma:cohom-no-augmenta} we have
$$|\pi_0(X^{A_0})|\leq b(X;\FF_p):=\sum_j b_j(X;\FF_p),$$
so
\begin{equation}
\label{eq:G_0:G_1}
[G_0:G_1]\leq b(X;\FF_p)!.
\end{equation}
Define the following sets of characters of $A_0$:
$$\Xi_E=\{\xi\in\wh{A_0}\mid E^{A_0,\xi}|_M\neq 0\},
\qquad
\Xi_N=\{\xi\in\wh{A_0}\mid N^{A_0,\xi}|_M\neq 0\}.$$
Let
$$m=\dim_{\RR}X=\rk_{\RR}E.$$
Since the decomposition of the fibers of $E|_M$ and $N|_M$
in irreducible representations of $A_0$ is locally constant
(because $A_0$ is finite), we may bound
$$|\Xi_E|\leq \frac{\rk_{\RR}(E)}{2}|\pi_0(M)|\leq \frac{m}{2}\cdot b(X;\FF_p),$$
and $|\Xi_N|$ can be bounded above by the same quantity. Since
$A_0$ is a normal subgroup of $G_1$, conjugation in $G_1$
induces an action of $G_1$ on $\wh{A_0}$ which preserves the
set $\Xi_E\cup\Xi_N$ because the action of $A_0$ on $E$ and $N$
extends to an action of $G_1$. Let $G_2\subseteq G_1$ be the
subgroup fixing each element of $\Xi_E\cup\Xi_N$. We have
\begin{equation}
\label{eq:G_1:G_2}
[G_1:G_2]\leq (m\cdot b(X;\FF_p))!.
\end{equation}
It follows from the definition that for any $\xi\in\Xi_E$ the action of
$G_2$ on $E|_M$ preserves the subbundle $E^{A_0,\xi}|_M\subseteq E|_M$,
and the same holds replacing $E$ by $N$.

Now let
$$\{V_1,\dots,V_r\}=\{E^{A_0,\xi}|_M\}_{\xi\in\Xi_E}\cup
\{N^{A_0,\xi}\}_{\xi\in\Xi_N}.$$ For each $j$ the action of
$G_2$ on $M$ lifts to an action on $V_j$, and we may identify
$\alpha$ with an element of $\cC(V_1,\dots,V_r)$. Let $\gamma$ be a
generator of the cyclic subgroup $H_2:=\pi(G_2)\subseteq H_0$,
let $g\in\pi^{-1}(\gamma)$ be any element, and let $K=\la
g\ra\subseteq G_2$, so that $\pi(K)=H_2$. Then
$(M,K,\{1\},V_1,\dots,V_r)$ is a $p^k$-admissible tuple.

Arguing similarly as in the proof of Lemma
\ref{lemma:grup-A-0}, using Theorems
\ref{thm:admissibility-hereditary} and \ref{thm:element-h-p-k}
if $k>1$, and Theorem \ref{thm:admissibility-hereditary} and
Lemma \ref{lemma:element-h-p} (with $\Phi=\{1\}$) if $k=1$, we
deduce the existence of an abelian subgroup $A_1\subseteq K$
satisfying $[K:A_1]\leq p^{k-1}$ and $M^{A_1}\neq\emptyset$.

Let $\Gamma=\la A_0,A_1\ra$. We have
$[G:\Gamma]=[B:B\cap\Gamma][H:\pi(\Gamma)]$. On the other hand,
$$[B:B\cap\Gamma]\leq [B:A_0]\leq p^{r(k-1)},$$
because $A_0\subseteq\Gamma$ and (3) in Lemma \ref{lemma:grup-A-0}. We have
$$[H_2:\pi(\Gamma)\cap H_2]\leq [K:A_1]$$
because $\pi:K\to H_2$ is a surjection and $\pi(A_1)\subseteq
\pi(\Gamma)\cap H_2$. Consequently if $k>1$ we may bound, using
in particular the first part of (\ref{eq:H:H_0}),
\begin{align*}
[H:\pi(\Gamma)] &\leq [H:\pi(\Gamma)\cap H_2] \\
&=[H:H_2][H_2:\pi(\Gamma)\cap H_2] \\
&\leq [H:H_0][H_0:H_2][K:A_1] \\
&\leq p^{kr^2}[G_0:G_2] p^{k-1} \\
&\leq p^{kr^2} b(X;\FF_p)!(m\cdot b(X;\FF_p))! p^{k-1},
\end{align*}
and hence $$[G:\Gamma]\leq p^{r(k-1)}p^{kr^2}
b(X;\FF_p)!(m\cdot b(X;\FF_p))! p^{k-1}.$$ If $k=0$ similar
arguments using the second part of (\ref{eq:H:H_0}) yield
$$[G:\Gamma]\leq b(X;\FF_p)!(m\cdot b(X;\FF_p))!.$$
Define $C_0'=1$ if
$\la[X],\e(E)\ra=1$, and if $\la[X],\e(E)\ra\neq 1$ define
$$C_0'=\sup_p \{p^{r(k-1)}p^{kr^2}
b(X;\FF_p)!(m\cdot b(X;\FF_p))! p^{k-1}\mid p^{k-1} \text{ divides
} \la[X],\e(E)\ra,\,k>1\},$$
where the supremum runs over the set of prime numbers.
This number is finite because the set of primes $p$ dividing $\la[X],\e(E)\ra$ is finite.
Define also
$$C''_0=\sup_p\{b(X;\FF_p)!(m\cdot
b(X;\FF_p))!\mid p\text{ does not divide $\la[X],\e(E)\ra$}\}.$$
This is also clearly finite, because $b(X;\FF_p)$ can be bounded above independently of $p$.
Both $C_0'$ and $C_0''$ are bounded above by constants depending only on $H^*(X)$ and
$\la[X],\e(E)\ra$ (recall that $m=\dim X$, so it can be recovered from $H^*(X)$).
Let
$$C_0=\max\{C_0',C_0''\}.$$
We have proved that any prime $p$ and any finite $p$-group $G$
acting on $E$ and sitting in an exact sequence as in
(\ref{eq:1-B-G-H-1}) there is a subgroup $\Gamma\subset G$ such
that $X^{\Gamma}\neq\emptyset$ and $[G:\Gamma]\leq C_0$. By
Corollary \ref{cor:Jordan-classic}, $\Gamma$ has an abelian
subgroup $A\subseteq\Gamma$ satisfying
$[\Gamma:A]\leq\Jor_{2m}$.
So, setting
$$C_1=\Jor_{2m}C_0,$$
we have proved that $G$ has an abelian subgroup $A$ of index
$[G:A]\leq C_1$ satisfying $X^A\neq\emptyset$.

\noindent{\bf Step 2.} Now suppose $G$ is arbitrary.
We claim that there exists some abelian subgroup $A_0\subseteq G$ satisfying
$[G:A_0]\leq C_1^{18r^5}$. Otherwise, Theorem \ref{thm:cyclic-by-abelian}
would imply the existence of subgroups $B\unlhd\Gamma\subseteq G$ with
$B$ abelian, $\Gamma/B$ cyclic, and $\alpha(\Gamma)>C_1$. However, the group
$\Gamma$ has the form considered in Step 1, so by the previous results we
must have $\alpha(\Gamma)\leq C_1$, a contradiction.

Now, if $A_0\subseteq G$ is an abelian subgroup, then by the result of
Step 1 there is a subgroup $A\subseteq A_0$ satisfying $X^A\neq\emptyset$
and $[A_0:A]\leq C_1$, so $[G:A]\leq C:=C_1^{1+18r^5}$.

\section{Lifting finite group actions to line bundles}
\label{s:lifting-actions-line-bundle}

In this section we prove Theorem \ref{thm:liftings}. Actually we divide
it in in three different statements: Theorems \ref{thm:lift-action-line-bundle},
\ref{thm:lift-Ham-action-line-bundle} and \ref{thm:lift-commutator-action-line-bundle}.

\subsection{The case of vanishing first Betti number}

\begin{theorem}
\label{thm:lift-action-line-bundle} Let $X$ be a connected,
smooth manifold satisfying $b_1(X)=0$ and let $L\to X$ be a
complex line bundle. Suppose that $G\subset\Diff(X)$ is a finite
subgroup satisfying $g^*L\simeq L$ for
every $g\in G$. Then there exists a finite group $G'$
sitting in a short exact sequence
$$1\to H\to G'\to G\to 1,$$
where $H$ is finite cyclic and $|H|=|G|$, and a smooth action of $G'$ on $L$ by
bundle automorphisms lifting the action of $G$ on $X$.
\end{theorem}
\begin{proof}
The following proof was kindly provided by the referee.
Consider the complex vector bundle
$$V=\bigoplus_{g\in G}(g^{-1})^*L.$$
We can identify the fiber of $V$ over $x$ with
$V_x=\{v:G\to L\mid v(g)\in L_{g^{-1}x}\}.$
Define an action of $G$ on $V$ lifting the action on $X$ as follows:
if $v\in V_x$ and $\gamma\in G$ then $\gamma v:G\to L$ is the map given by
$\gamma v(g):=v(\gamma^{-1}g)$, so that $\gamma v\in V_{\gamma x}$.
The action of $G$ on $V$ induces an action on
$\det V=\bigotimes_{g\in G}(g^{-1})^*L.$
Choosing for every $g\in G$ an isomorphism $(g^{-1})^*L\simeq L$
we obtain an isomorphism $\det V\simeq L^{\otimes d}$, where $d=|G|$.
We have thus defined an action of $G$ on $L^{\otimes d}$ lifting the action on $X$.
This action is encoded in the assignment to
every $(g,x)\in G\times X$ of an isomorphism
$\phi_{g,x}:L_x^{\otimes d}\stackrel{\simeq}{\longrightarrow} L_{gx}^{\otimes d}$ varying continuously with $x$.
Let $I_{g,x}$ denote the set of the $d$ different isomorphisms $\psi:L_x\to L_{gx}$ satisfying
$\psi^{\otimes d}=\phi_{g,x}$. Then
$$I_g=\bigcup_{x\in X}I_{g,x}$$
is a subset of the (total space of the) line bundle $L^{\vee}\otimes g^*L$.
Endowing $I_g$ with the subspace topology, the projection $I_g\to X$ becomes a principal $\mu_d$-bundle, where $\mu_d\subset S^1$ is the group of $d$-th roots of unity. Its topology is
therefore encoded in a monodromy morphism $\pi_1(X)\to\mu_d$. Since by assumption $b_1(X)=0$
the monodromy is trivial, so $I_g\to X$ can be trivialized. Now let
$$G'=\{(g,s)\mid g\in G,\,s:X\to I_g\text{ a section of $I_g\to X$}\}$$
and consider the action of $G'$ on $L$ defined as follows: if $g'=(g,s)\in G'$ and $x\in X$, then
the action of $g'$ sends $L_x$ to $L_{gx}$ via the map $s(x)$. This defines an injective map
$G'\to\Diff(L)$ with identifies $G'$ with a subgroup of $\Diff(L)$ and hence endows $G'$ with
a group structure. Clearly,
the projection $G'\ni (g,s)\mapsto g\in G$ is a surjective morphism
of groups with respect to which the action of $G'$ on $L$ lifts the action of $G$ on $X$,
and the kernel of $G'\to G$ can be identified with $\mu_d$ because $X$ is connected.
\end{proof}

\subsection{A criterion for existence of actions of central extensions on line bundles}

Recall (see e.g.
\cite[\S III.1]{B}) that for a group $G$ and a left
$G$-module $M$ the bar complex of $G$ with coefficients in $M$
has $n$-th term $C^n(G,M)=\Map(G^n,M)$ and the coboundary map
$\delta:C^{n-1}(G,M)\to C^n(G,M)$ sends $f:G^{n-1}\to M$ to the
map $\delta f:G^n\to M$ defined as
\begin{equation}
\label{eq:cocycle}
(\delta f)(g_1,\dots,g_n)=g_1f(g_2,\dots,g_n)-f(g_1g_2,\dots,g_n)+\dots
+(-1)^nf(g_1,\dots,g_{n-1}).
\end{equation}
Then $H^*(G,M)$ is the cohomology of the complex
$(C^*(G,M),\delta)$.
If $N$ is another $G$-module then
any morphism $\phi:M\to N$ of $G$-modules induces a morphism in
cohomology $\phi_*:H^*(G,M)\to H^*(G,N)$.

Assume that $X$ is a connected
smooth manifold,
$G\subset\Diff(X)$ is a finite subgroup, and
$L\to X$ is a complex line
bundle satisfying $L\simeq g^*L$ for every $g\in G$. Take
a Hermitian metric $|\cdot|$ on $L$.

Choose for every $g\in G$ a bundle automorphism $\beta_g:L\to
L$ lifting $g^{-1}$ and satisfying
$|\beta_g(\lambda)|=|\lambda|$ for every $\lambda\in L$ (that
$\beta_g$ exists is a consequence of the assumption $L\simeq
(g^{-1})^*L$). We call $\beta=(\beta_g)_{g\in G}$ a {\it
set of lifts}.

For every $g_1,g_2\in G$ the composition
$\beta_{g_1g_2}^{-1}\beta_{g_2}\beta_{g_1}$ is a bundle
automorphism of $L$ lifting the identity and preserving the
Hermitian metric, and hence can be identified with pointwise
multiplication by a function
$f_{\beta}(g_1,g_2)\in\cC^{\infty}(X,S^1)$. To simplify the
notation we use this convention: if $c\in\cC^{\infty}(X,S^1)$
is any map, we denote by the same symbol $c$ the map $L\to L$
given by pointwise multiplication by $c$. For example, we write
$\beta_{g_1g_2}^{-1}\beta_{g_2}\beta_{g_1}=
f_{\beta}(g_1,g_2)$.
Consider the $G$-action on $\cC^{\infty}(X,S^1)$ defined as
$$(g\cdot\phi)(x)=\phi(g^{-1}\cdot x)$$
for any $g\in G$, $\phi\in \cC^{\infty}(X,S^1)$ and $x\in X$. For any
$c\in\cC^{\infty}(X,S^1)$ we have
\begin{equation}
\label{eq:conjugation-L}
\beta_{g}^{-1}c\beta_{g}=g \cdot c.
\end{equation}

Clearly, the set of lifts $\beta$ defines a lift of the action
of $G$ to $L$ (for which the action of $g$ is given by the map
$\beta_{g^{-1}}$) if and only if $f_{\beta}(g_1,g_2)\equiv 1$
for every $g_1,g_2$. The following identity shows that
$f_{\beta}$ is cocycle in $C^2(G,\cC^{\infty}(X,S^1))$ (we use multiplicative notation
for cochains with values in $\cC^{\infty}(X,S^1)$, instead of
additive notation as in (\ref{eq:cocycle})):
$$\overbrace{\beta_{g_1g_2g_3}^{-1}\beta_{g_3}\beta_{g_1g_2}}^{f_{\beta}(g_1g_2,g_3)}
\overbrace{\beta_{g_1g_2}^{-1}\beta_{g_2}\beta_{g_1}}^{f_{\beta}(g_1,g_2)}=
\beta_{g_1g_2g_3}^{-1}\beta_{g_3}\beta_{g_2}\beta_{g_1}=
\overbrace{\beta_{g_1g_2g_3}^{-1}\beta_{g_2g_3}\beta_{g_1}}^{f_{\beta}(g_1,g_2g_3)}
\overbrace{\beta_{g_1}^{-1}\beta_{g_2g_3}^{-1}\beta_{g_3}\beta_{g_2}\beta_{g_1}}^{g_1\cdot f_{\beta}(g_2,g_3)}.$$

If $\beta'=(\beta'_g=\beta_gc_g)_{g\in G}$ is another set of lifts (with $c_g\in\cC^{\infty}(X,S^1)$) then
\begin{align*}
f_{\beta'}(g_1,g_2) &=
{\beta'}_{g_1g_2}^{-1}\beta'_{g_2}\beta'_{g_1}
= c_{g_1g_2}^{-1}\beta_{g_1g_2}^{-1}\beta_{g_2}c_{g_2}\beta_{g_1}c_{g_1} \\
&= \beta_{g_1g_2}^{-1}\beta_{g_2}\beta_{g_1}(g_1\cdot c_{g_2})c_{g_1}, \qquad\qquad\text{by (\ref{eq:conjugation-L}),} \\
&= \beta_{g_1g_2}^{-1}\beta_{g_2}\beta_{g_1}c_{g_1g_2}^{-1}(g_1\cdot c_{g_2})c_{g_1}, \qquad\text{because
$\beta_{g_1g_2}^{-1}\beta_{g_2}\beta_{g_1}$ lifts the identity on $X$}, \\
&= f_{\beta}(g_1,g_2)c_{g_1g_2}^{-1}(g_1\cdot c_{g_2})c_{g_1}.
\end{align*}
So if $c\in C^1(G,\cC^{\infty}(X,S^1))$ is the cochain defined
as $c(g)=c_g$, then $f_{\beta'}=f_{\beta}\,\delta c$
(multiplicative notation!), so $f_{\beta}$ and $f_{\beta'}$ are
cohomologous.

It follows that the cohomology class
$[f_{\beta}]\in H^2(G,\cC^{\infty}(X,S^1))$ is independent of
the choice of $\beta$, and that the action of $G$ on $X$ lifts
to an action on $L$ if and only if $[f_{\beta}]=0$.
Now define
$$b(L)\in H^2(G,H^1(X;\ZZ))$$
to be the image of $[f_{\beta}]$ under the map in cohomology induced by
passing to homotopy classes $\cC^{\infty}(X,S^1)\to [X,S^1]=H^1(X;\ZZ)$, where
$H^1(X;\ZZ)$ is endowed with the action of $G$ inherited from that on $\cC^{\infty}(X,S^1)$.

\begin{theorem}
\label{thm:existence-lifting}
Then there exists a finite group $G'$
sitting in a short exact sequence
$$1\to H\to G'\to G\to 1,$$
where $H$ is finite cyclic, and a smooth action of $G'$ on $L$ by
bundle automorphisms lifting the action of $G$ on $X$, if and only if $b(L)=0$.
Furthermore, if $G'$ exists, then it can be chosen so that
$|H|=|G|$.
\end{theorem}

Note that Theorem \ref{thm:existence-lifting} includes Theorem \ref{thm:lift-action-line-bundle} as a particular case.

Although we will not use this fact here, it is worth pointing out that
$b(L)$ may be interpreted as the image of $c_1(L)$ via the
differential
$$d_2^{0,2}:H^2(X;\ZZ)^G\to H^2(G,H^1(X;\ZZ))=H^2(BG,H^1(X;\ZZ))$$
in the Serre spectral sequence for the fibration $X_G\to BG$. Hence
the vanishing of $b(L)$ is the first obstruction to lifting $c_1(L)$ to
an equivariant cohomology class. The second and last obstruction is the
vanishing of the map $d_3^{0,2}:\Ker b\to H^3(G,\ZZ)$, which also admits a geometric
interpretation: if $b(L)=0$, then $d_3^{0,2}(c_1(L))\in H^3(G,\ZZ)\simeq H^2(G,S^1)$
defines a central extension $1\to S^1\to \GG\to G\to 1$, and the group $G'$ in
Theorem \ref{thm:existence-lifting} can be identified with a subgroup of $\GG$.

\begin{proof}
Suppose that there is an extension $1\to H\to G'\to G\to 1$ with the properties in the
statement of the theorem, and that $G'$ acts on $L$ lifting the action of $G$ on $X$. Then
the action of $H$ on $L$ is necessarily given by constant maps $X\to S^1$.
Choose for every $g\in G$ a lift $g'\in G'$ and let $\beta_g:L\to L$ be
the map given by the action of $(g')^{-1}$. Then $\beta_{g_1g_2}^{-1}\beta_{g_2}\beta_{g_1}$
is given by the action of some element of $H$, and hence is a constant map.
It follows that $b(L)=0$.

Now assume, conversely, that $b(L)=0$. It follows from the discussion before the theorem
that we may choose a set of lifts $\beta=(\beta_g)$ such that $\beta_{g_1g_2}^{-1}\beta_{g_2}\beta_{g_1}$ is homotopic to a constant map for every $g_1,g_2$.
This is equivalent to the existence of
a function
$\phi_{\beta}(g_1,g_2)\in\cC^{\infty}(X,\RR)$
such that
$$f_{\beta}(g_1,g_2)=\exp(2\pi\imag \phi_{\beta}(g_1,g_2)).$$

Since $X$ is connected, the cocycle condition $\delta f_{\beta}=0$
implies that
\begin{equation}
\label{eq:almost-zero}
g_1\cdot\phi_{\beta}(g_2,g_3)-\phi_{\beta}(g_1g_2,g_3)+\phi_{\beta}(g_1,g_2g_3)-\phi_{\beta}(g_1,g_2)\quad\text{is
a constant integer.}
\end{equation}
Since this integer need not be zero, in general the cochain
$\phi_{\beta}$ defined by the functions $\phi_{\beta}(g_1,g_2)$
is not a cocycle in $C^2(G,\cC^{\infty}(X,\RR))$.

Let $x_0\in X$ be any point and define
$$\cC_0^{\infty}(X,\RR)=\{f\in\cC^{\infty}(X,\RR)\mid \sum_{g\in G}f(g\cdot
x_0)=0\}.$$ Consider the linear projection
$\Pi:\cC^{\infty}(X,\RR)\to \cC_0^{\infty}(X,\RR)$ defined as
$$\Pi(f)=f-\frac{1}{|G|}\sum_{g\in G}f(g\cdot x_0).$$
The subspace $\cC_0^{\infty}(X,\RR)\subset \cC^{\infty}(X,\RR)$
is $G$-invariant, the map $\Pi$ is $G$-equivariant, and the
kernel of $\Pi$ consists of the constant functions. So if we
define $\phi_{\beta}^0(g_1,g_2):=\Pi(\phi_{\beta}(g_1,g_2))$
then (\ref{eq:almost-zero}) implies that
$$g_1\cdot\phi_{\beta}^0(g_2,g_3)-\phi_{\beta}^0(g_1g_2,g_3)+\phi_{\beta}^0(g_1,g_2g_3)-\phi_{\beta}^0(g_1,g_2)=0,$$
i.e., $\phi_{\beta}^0$ is a cocycle in
$C^2(G,\cC^{\infty}_0(X,\RR))$. Since $\cC^{\infty}_0(X,\RR)$
is an $\RR$-vector space and $G$ is finite,
$H^*(G,\cC^{\infty}_0(X,\RR))=0$ (see e.g. \cite[Ch. III, Cor.
10.2]{B}). It follows that there exists some $\kappa\in
C^1(G,\cC^{\infty}_0(X,\RR))$ such that $\phi_{\beta}^0=\delta
\kappa$. So if we replace each
$\beta_g$ by $\beta_g\exp(-2\pi\imag \kappa(g))$ then the new
set of lifts, which we still denote by $\beta=(\beta_g)_{g\in
G}$, has the property that $f_{\beta}(g_1,g_2)$ is a constant
function $X\to S^1$ for every $g_1,g_2\in G$. In other words,
we may now regard $f_{\beta}$ as defining a cocycle in
$C^2(G,S^1)$.

\begin{lemma}
\label{lemma:roots-of-unity} Let $d=|G|$ and let $\mu_d\subset
S^1$ be the subgroup of $d$-th roots of unity. We regard
$\mu_d$ as a $G$-module with trivial $G$-action. The inclusion
$\mu_d\subset S^1$ induces for every $k>0$ a surjective map
$H^k(G,\mu_d)\exh H^k(G,S^1)$.
\end{lemma}
\begin{proof}
Let $e:\RR\to S^1$ be the map $e(t)=\exp(2\pi\imag t)$.
Consider the following commutative diagram of trivial
$G$-modules with exact rows
$$\xymatrix{0 \ar[r] & \ZZ \ar@{^{(}->}[r] \ar@{=}[d] & \frac{1}{d}\ZZ \ar[r]^{e} \ar@{^{(}->}[d] & \mu_d\ar@{^{(}->}[d] \ar[r] & 1 \\
0\ar[r] & \ZZ\ar@{^{(}->}[r] & \RR\ar[r]^{e} & S^1\ar[r] &
1.}$$ Let $k>0$ be an integer. Taking cohomology we obtain the
following commutative diagram with exact rows (see e.g.
\cite[Ch. III, Prop. 6.1]{B}):
$$\xymatrix{H^k(G,\frac{1}{d}\ZZ) \ar[d]\ar[r] & H^k(G,\mu_d)
\ar[d]\ar[r]^{\partial_{\mu}} & H^{k+1}(G,\ZZ) \ar@{=}[d]\ar[r]^{\iota_d} & H^{k+1}(G,\frac{1}{d}\ZZ) \ar[d] \\
H^k(G,\RR) \ar[r] & H^k(G,S^1) \ar[r]^{\partial_S} &
H^{k+1}(G,\ZZ) \ar[r] & H^{k+1}(G,\RR).}$$ As before, \cite[Ch.
III, Cor. 10.2]{B} implies that $H^k(G,\RR)=H^{k+1}(G,\RR)=0$,
so $\partial_S$ is an isomorphism. Hence we need to prove that
$\partial_\mu$ is surjective, and this will follow if we prove
that $\iota_d=0$. The vanishing of $\iota_d$ follows again from
\cite[Ch. III, Cor. 10.2]{B}. Indeed, if $a\in C^{k+1}(G,\ZZ)$
is a cocycle, then $\frac{a}{d}\in C^{k+1}(G,\frac{1}{d}\ZZ)$
is also a cocycle, and we have $\iota_d[a]=d[a/d]$. Since
$d=|G|$, \cite[Ch. III, Cor. 10.2]{B} implies that $d[a/d]=0$.
\end{proof}

We have so far constructed a set of lifts $\beta$ with the
property that for every $g_1,g_2$ the function
$f_{\beta}(g_1,g_2):X\to S^1$ is constant. By Lemma
\ref{lemma:roots-of-unity}, there exists some $\lambda\in
C^1(G,S^1)$ such that $(f_{\beta}\,\delta\lambda)(g_1,g_2):X\to
S^1$ is constant and its value is a $d$-th root of unity. We
now replace each $\beta_g$ by $\beta_g\lambda(g)$, and denote
again by $\beta_g$ the resulting map. Then we have
$f_{\beta}(g_1,g_2)\in\mu_d$ for every $g_1,g_2$.

We claim that for every
$g_1,\dots,g_k\in G$ there is some $\theta\in\mu_d$ such that
$$\beta_{g_1\dots g_k}^{-1}\beta_{g_1}\dots\beta_{g_k}=\theta.$$
We prove the claim using induction on $k$. The case $k=1$ is trivial, and the case $k=2$
is the statement
that $f_{\beta}(g_1,g_2)\in\mu_d$ for every $g_1,g_2$.
If $k>2$ and the claim is true for
lower values of $k$, then we may write
$$\beta_{g_1}\dots\beta_{g_{k-1}}=\beta_{g_1\dots
g_{k-1}}\theta'$$ for some $\theta'\in\mu_d$. Then
\begin{align*}
\beta_{g_1\dots g_k}^{-1}\beta_{g_1}\dots\beta_{g_k} &=
\beta_{g_1\dots g_k}^{-1}\beta_{g_1\dots
g_{k-1}}\theta'\beta_{g_k} \\
&= \beta_{g_1\dots
g_k}^{-1}\beta_{g_1\dots g_{k-1}}\beta_{g_k}\theta'
= f_{\beta}(g_1\dots g_{k-1},g_k)\theta'\in\mu_d,
\end{align*}
which proves the induction step and hence the claim.

At this point we define $G'\subset\Diff(L)$ to be the subgroup
generated by $\{\beta_g\mid g\in G\}$. By construction $G'$ acts
on $L$ by bundle automorphisms. Consider the surjective
morphism
$$\pi:G'\to G$$
satisfying $\pi(\beta_g)=g$. To conclude the proof of the theorem
we prove that any element in
$\Ker\pi$ can be identified with multiplication by an element of $\mu_d$. Now, if
$g_1,\dots,g_k\in G$ satisfy $g_1\dots g_k=\Id$ then by the previous claim
there exists some $\theta\in\mu_d$ such that
$$\beta_{g_1}\dots\beta_{g_k}=\beta_{g_1\dots g_k}\theta=\beta_{\Id}\theta.$$
Since $f_{\beta}(\Id,\Id)=\beta_{\Id}^{-1}\beta_{\Id}\beta_{\Id}=\beta_{\Id}\in \mu_d$, it follows that $\beta_{g_1}\dots\beta_{g_k}\in\mu_d$.
\end{proof}

\subsection{Lifting Hamiltonian finite group actions}

\begin{theorem}
\label{thm:lift-Ham-action-line-bundle} Let $(X,\omega)$ be a
compact and connected symplectic manifold and let $L\to X$ be a
complex line bundle. For any finite subgroup
$G\subset\Ham(X,\omega)$
there is a finite group $G'$ sitting in a short exact sequence
$$1\to H\to G'\to G\to 1,$$
where $H$ is finite and cyclic and $|H|$ divides $|G|$, and a smooth action of $G'$
on $L$ by bundle automorphisms lifting the action of $G$ on
$X$.
\end{theorem}
\begin{proof}
Take a Hermitian structure on $L$ and let $\nabla$ be a unitary
connection on $L$.

Choose for each $g\in G$ a time dependent Hamiltonian
$H_g:X\times [0,1]\to\RR$ whose flow at time $1$ is equal to
$g$. Reparametrizing $H_g$ in the time coordinate (and
rescaling accordingly) if necessary, we may assume that the
support of each $H_g$ is contained in $X\times (a,b)$ for some
open subinterval $(a,b)\subset [0,1]$.

Denote by $\beta_g:L\to
L$ the bundle automorphism obtained using parallel transport
with respect to $\nabla$ along the paths given by the
Hamiltonian flow of $H_{g^{-1}}$. More precisely, let
$\{\phi_{g,t}\}_{t\in [0,1]}$ denote the path in
$\Ham(X,\omega)$ defined by the conditions $\phi_{g,0}=\Id_X$
and, for each $t\in [0,1]$ and $x\in X$,
$$\left.\frac{\partial}{\partial
t}\phi_{g,t}(x)\right|_{t=\tau}=\xX_{g,\tau}(\phi_{g,\tau}(x)),$$
where $\xX_{g,t}$ is the Hamiltonian vector field associated to
$H_g(\cdot,t)$, which means that
\begin{equation}
\label{eq:transport-parallel-nabla}
d(H_g(\cdot,t))=\iota_{\xX_{g,t}}\omega,
\end{equation}
where $\iota$ denotes the standard contraction map. By our
choice of $H_g$ we have $g=\phi_{g,1}$.

Let $\Pi:L\to X$ be the projection map.
For each $g\in G$ and each $\lambda\in L$ there is a unique
smooth map $\psi_g^\lambda:[0,1]\to L$ satisfying
$\psi_g^{\lambda}(0)=\lambda$,
$\Pi(\psi_g^{\lambda}(t))=\phi_{g,t}(x)$, with
$x=\Pi(\lambda)$, and $(\psi_g^{\lambda})'(t)$ belongs to the
horizontal distribution of $\nabla$ for each $t$. Define
$$\beta_g(\lambda):=\psi_{g^{-1}}^{\lambda}(1).$$
Then $\beta_g$ lifts $g^{-1}$ and preserves the
Hermitian structure on $L$. Consider, as in the previous subsection,
the cocycle
$f_{\beta}\in C^2(G,\cC^{\infty}(X,S^1))$ defined by
$$f_{\beta}(g_1,g_2)=\beta_{g_1g_2}^{-1}\beta_{g_2}\beta_{g_1}$$ for every
$g_1,g_2\in G$. We claim that for every $g_1,g_2\in G$
the map $f_{\beta}(g_1,g_2):X\to S^1$ is
null-homotopic (i.e. homotopic to a constant map). By Theorem \ref{thm:existence-lifting},
the claim implies the present theorem.
On the other hand, the claim is equivalent to the condition that
for any smooth map $\gamma:S^1\to X$ the composition
$f_{\beta}(g_1,g_2)\circ\gamma:S^1\to S^1$ is null-homotopic, or equivalently that
the degree of $f_{\beta}(g_1,g_2)\circ\gamma$ is $0$.

Fix elements $g_1,g_2\in G$ and a smooth map $\gamma:S^1\to X$.
To compute $\deg (f_{\beta}(g_1,g_2)\circ\gamma)$ we consider
the concatenation of the flows defining $\alpha_{g_1^{-1}}$,
$\alpha_{g_2^{-1}}$ and $\alpha_{g_2^{-1}g_1^{-1}}$. Namely let
$Z:X\times [0,1]\to \RR$ be defined as follows:
$$Z(x,t)=\left\{\begin{array}{ll}
3H_{g_1^{-1}}(x,3t) & \quad\text{if $t\in [0,\frac{1}{3}]$,}\\
3H_{g_2^{-1}}(x,3t-1) & \quad\text{if $t\in [\frac{1}{3},\frac{2}{3}]$,}\\
3H_{g_2^{-1}g_1^{-1}}(x,3t-2) & \quad\text{if $t\in [\frac{2}{3},1]$.}
\end{array}\right.$$
Note that $Z$ is smooth because of our assumption on the support of each $H_g$.
Let $\{\Phi_t\}_{t\in [0,1]}$ be the path in $\Ham(X,\omega)$
defined by $Z$, which we view as a time depending Hamiltonian.
Then $\Phi_1(x)=x$ for every $x\in X$, so we may define a map
$$\Psi:S^1\times S^1\to X,\qquad \Psi(e^{2\pi\imag u},e^{2\pi\imag
v})=\Phi_u(\gamma(e^{2\pi\imag v})).$$
We can identify $f_{\beta}(g_1,g_2)(\gamma(\theta))$ with the
holonomy of $(\Psi^*L,\Psi^*\nabla)$ along
$S^1\times\{\theta\}$ with the orientation given by the path
$[0,1]\ni u\mapsto (e^{2\pi\imag u},\theta)\in S^1\times S^1$.
Define $\eta:[0,1]\to S^1$ by
$\eta(v)=f_{\beta}(g_1,g_2)(\gamma(e^{2\pi\imag v}))$. Let
$[S^1\times S^1]\in H_2(S^1\times S^1)$ denote the
fundamental class, defined taking the counterclockwise
orientation on each factor. We have:
\begin{align*}
\deg f_{\beta}(g_1,g_2)\circ \gamma &= \frac{1}{2\pi\imag}
\int_0^1 \frac{\eta'(v)}{\eta(v)}\,dv \\
&= \frac{\imag}{2\pi}\int_{S^1\times S^1}\Psi^*F_{\nabla} \qquad\text{by Stokes theorem} \\
&= \la c_1(\Psi^*L),[S^1\times S^1]\ra \qquad\text{by Chern--Weil theory} \\
&= \la c_1(L),\Psi_*[S^1\times S^1]\ra.
\end{align*}
Now, Theorem 1.A in \cite{LMP} implies that $\Psi_*[S^1\times
S^1]\in H_2(X;\QQ)$ is zero; indeed, $\Phi$ defines a loop in
$\Ham(X,\omega)$, and $\Psi_*[S^1\times S^1]$ is equal to
$\partial_{\Phi}(\gamma_*[S^1])$ in the notation of \cite{LMP}
(the proof of Theorem 1.A in \cite{LMP} requires $X$
to be spherically monotone; the proof for a general compact
symplectic manifold is given in \cite{McD}). It follows that
$\deg (f_{\beta}(g_1,g_2)\circ \gamma)=0$, so the proof of the
claim is complete.
\end{proof}

\subsection{Lifting actions of commutators of finite groups}

\begin{theorem}
\label{thm:lift-commutator-action-line-bundle} Let $X$ be a
compact and connected smooth manifold and let $L\to X$ be a
complex line bundle. Suppose that $\Gamma\subset\Diff(X)$ is a finite
subgroup satisfying $\gamma^*L\simeq L$ for every $\gamma\in\Gamma$,
and suppose that the action of $\Gamma$ on $H^1(X)$ is trivial.
There is a finite group $G'$ sitting in a short exact sequence
$$1\to H\to G'\to [\Gamma,\Gamma]\to 1,$$
where $H$ is finite and cyclic and $|H|$ divides $|[\Gamma,\Gamma]|$, and a smooth action of $G'$
on $L$ by bundle automorphisms lifting the action of $G$ on $X$.
\end{theorem}
\begin{proof}
Let $G=[\Gamma,\Gamma]$. By Theorem \ref{thm:existence-lifting} it suffices to construct
a set of lifts $(\beta_g)_{g\in G}$ such that $\beta_{g_1g_2}^{-1}\beta_{g_2}\beta_{g_1}$
is homotopic to a constant map for every $g_1,g_2\in G$.

Let $\eta_0\in\Omega^2(X)$ be de Rham representative of $c_1(L)$.
For every $\gamma\in\Gamma$ we have $\gamma^*L\simeq L$,
and this implies that $\gamma^*\eta_0$ is cohomologous to $\eta_0$.
Consequently
$$\eta=\frac{1}{|\Gamma|}\sum_{\gamma\in\Gamma}\gamma^*\eta_0$$
is $\Gamma$-invariant and also represents $c_1(L)$.
Take a Hermitian metric on $L$, and let
$\aA_\eta$ denote the space of all unitary connections on $L$ whose curvature is equal to
$-\imag 2\pi\eta$. By de Rham's theorem $A_{\eta}$ is nonempty.
Addition of closed one forms gives $\aA_{\eta}$ the
structure of torsor over
$\Ker(d:\Omega^1(X,\imag\RR)\to\Omega^2(X,\imag\RR))$.
Let
$$\mM_{\eta}=\aA_{\eta}/\Map(X,S^1),$$
where elements of $\Map(X,S^1)$ act on $\aA_{\eta}$ as gauge transformations of $L$.
The torsor structure on $\aA_{\eta}$ induces on $\mM_{\eta}$ a struture of torsor over
$$T:=\frac{\Ker(d:\Omega^1(X,\imag\RR)\to\Omega^2(X,\imag\RR))}{d\log (\Map(X,S^1))}\simeq
\frac{H^1(X;\RR)}{H^1(X,\ZZ)}\simeq (S^1)^{b_1(X)},$$
where $d\log$ sends $\sigma:X\to S^1$ to $\sigma^{-1}d\sigma\in
\Ker(d:\Omega^1(X,\imag\RR)\to\Omega^2(X,\imag\RR))$.

The group $\Gamma$ acts on $\mM_{\eta}$ by pullback: if $[A]\in\mM_{\eta}$ is the gauge
equivalence class of a connection $A$ and $\gamma\in\Gamma$, then $\gamma^*A$ is a connection
on $\gamma^*L$, so if $\iota:L\to\gamma^*L$ is an isomorphism of Hermitian line bundles
(which exists, since $L$ and $\gamma^*L$ are isomorphic line bundles)
then $\iota^*\gamma^*A$ is a connection on $L$, whose gauge equivalence class is independent
of $\iota$. Hence we may define
$$\gamma^*[A]:=[\iota^*\gamma^*A].$$
This action is compatible with the torsor structure, in the sense that if $\tau\in T$ then
$\gamma^*(\tau\cdot[A])=\tau\cdot\gamma^*[A]$ for every $\gamma$ and $[A]$. Indeed, one may
choose as a representative of $\tau$ a closed $1$-form $\alpha$, so that $\tau\cdot[A]=[A+\alpha]$,
and compute
\begin{align*}
\gamma^*(\tau\cdot[A]) &=\gamma^*[A+\alpha]=[\gamma^*A+\gamma^*\alpha]=
[\gamma^*A+(\gamma^*\alpha-\alpha)+\alpha] \\
&=\tau\cdot[\gamma^*A+(\gamma^*\alpha-\alpha)]
=\tau\cdot[\gamma^*A]=\tau\cdot\gamma^*[A].
\end{align*}
Here we have used the fact that $\gamma^*\alpha-\alpha$ is exact, which follows from the
assumption that $\Gamma$ acts trivially on $H^1(X)$.

It follows that for every $\gamma\in\Gamma$ there exists some $\tau_{\gamma}\in T$ such
that $\gamma^*[A]=\tau_{\gamma}\cdot[A]$ for every $[A]\in\mM_{\eta}$. The map
$\Gamma\to T$ sending $\gamma$ to $\tau_{\gamma}$ is necessarily a morphism of groups,
and since $T$ is abelian its restriction to $G$ is trivial. Hence
$G$ acts trivially on $\mM_{\eta}$.

Fix a connection $A\in\aA_{\eta}$.
For every $g\in G$ let $\beta_g:L\to L$ be an isomorphism lifting $g^{-1}$ and satisfying
$\beta_g^*A=A$. Such $\beta_g$ exists because $G$ acts trivially on $\mM_{\eta}$.
Then, for every $g_1,g_2\in G$,
$\beta_{g_1g_2}^{-1}\beta_{g_2}\beta_{g_1}$ is a gauge transformation of $L$
which fixes the connection $A$. This implies that $\beta_{g_1g_2}^{-1}\beta_{g_2}\beta_{g_1}$
is constant as a map from $X$ to $S^1$, and so $(\beta_g)_{g\in G}$ is the required set of lifts.
\end{proof}

\section{Proofs of the main theorems}
\label{s:proofs-main}

\subsection{Proof of Theorem \ref{thm:fixed-pt-symplectic}}
\label{ss:proof-thm:fixed-pt-symplectic}
Take a class $w'\in H^*(X;\QQ)$ sufficiently close to $w$
so that $(w')^n\neq 0$. Multiplying $w'$ by a suitable integer
we obtain an integral class $v\in H^2(X)$ satisfying $v^n\neq 0$.
Let $L\to X$ be a complex line bundle such that $c_1(L)=v$.
Let $E=L^{\oplus n}$. Then $\e(E)=c_1(L)^n=v^n$, so
$\la [X],\e(E)\ra\neq 0$.

Let $p$ be a prime and let $G$ be a finite $p$-group acting on $X$ preserving an
almost complex structure. Since $b_1(X)=0$, by Theorem \ref{thm:lift-action-line-bundle}
there exists a short exact sequence
$$1\to H\to G_L\stackrel{\pi}{\longrightarrow} G\to 1$$
such that $G_L$ is a finite $p$-group and the action of $G$ on $X$ lifts to an action
of $G_L$ on $L$. Let $G'$ be the fiber product of $n$ copies of $G_L$ over $G$, i.e.,
$$G'=\{(g_1,\dots,g_n)\in G_L^n\mid \pi(g_1)=\dots=\pi(g_n)\},$$
and let $\pi':G'\to G$ be the map $\pi'(g_1,\dots,g_n)=\pi(g_1)$.
The action of $G_L$ on $L$ induces an action of $G'$ on $E$ lifting the
action of $G$ on $X$. Since the action of $G$ on $X$ preserves $J$,
the action of $G'$ on
$E$ is given by a morphism $G'\to\Aut(E,X,J)$.

Applying Theorem \ref{thm:main-fixed-point} to the image of the morphism
$G'\to\Aut(E,X,J)$ we deduce the
existence of an abelian subgroup $A'\subseteq G'$ such that $X^{A'}\neq\emptyset$ and
$[G':A']\leq C$. It follows that $A:=\pi'(A')$ satisfies $X^A\neq\emptyset$ and $[G:A]\leq C$.

\subsection{Proof of Theorem \ref{thm:fixed-pt-hamiltonian}}
The hypothesis that $G$ is a finite subgroup of $\Ham(X,\omega)$ implies that the action of $G$ on
$X$ preserves an almost complex structure on $X$ (see Remark \ref{rmk:ac-structures}).
Then the proof is almost the same as that of Theorem \ref{thm:fixed-pt-symplectic} which we just
gave, replacing
Theorem \ref{thm:lift-action-line-bundle} by Theorem \ref{thm:lift-Ham-action-line-bundle}.

\subsection{A criterion to test Jordan property}

Suppose that
$\cC$ is a set of finite groups and let $C,d$ be natural numbers.
We say that $\cC$ satisfies the property $\jJ(C,d)$
if each $G\in\cC$ has an abelian subgroup $A$ such that $[G:A]\leq C$ and
$A$ can be generated by $d$ elements. Let $\tT(\cC)$ the set of all $T \in \cC$ such that
there exist distinct primes $p$ and $q$, a Sylow $p$-subgroup $P$ of $T$ (which might
be trivial), and a normal Sylow $q$-subgroup $Q$ of $T$, such that $T = PQ$.
The following is the main result in \cite{MT}. Note that the proof uses the classification of
finite simple groups.

\begin{theorem}
\label{thm:MT}
If $\cC$ is closed under taking subgroups and
if $\tT(\cC)$ satisfies $\jJ(C,d)$
then $\cC$ satisfies $\jJ(C',d)$ for some natural number $C'$
depending on $C$ and $d$.
\end{theorem}

Theorem \ref{thm:MT} is false if one replaces $\tT(\cC)$ by the collection
of all $p$-groups in $\cC$ for varying prime $p$. However, for certain collections
of finite groups it does suffice to consider only the $p$-groups in order to obtain
property $\jJ(C,d)$, as the following theorem shows.

\begin{theorem}
\label{thm:test-Jordan} Let $X$ be a $2n$-dimensional smooth,
compact and connected manifold. Let $\cC$ be a collection of finite
subgroups of $\Diff(X)$. Suppose that:
\begin{enumerate}
\item $\cC$ is closed under taking subgroups,
\item for any $G\in\cC$ there is a $G$-invariant almost complex structure
on $X$,
\item there exists a constant $C_0$ such
that, for any prime $q$, any $G\in\cC$ which is a $q$-group
has a subgroup $G_0\subseteq G$ satisfying
$X^{G_0}\neq\emptyset$ and $[G:G_0]\leq C_0$.
\end{enumerate}
Then $\cC$ satisfies $\jJ(C,n)$, where $C$ only
depends on $C_0$ and $H^*(X)$.
\end{theorem}
\begin{proof}
This statement is implicit in the proof of Theorem 5.2 in \cite{M3}, so we just
sketch the main ideas and refer to \cite{M3} for more details.

By Theorem \ref{thm:MT} we only need to prove that $\tT(\cC)$ satisfies $\jJ(C,n)$
for some $C$ depending only on $C_0$ and $H^*(X)$.

Let us say that some quantity is $H^*(X)$-bounded if it admits
an upper bound which depends only on $H^*(X)$. Of course, $2n$
is $H^*(X)$-bounded.

Let $T\in\tT(\cC)$ and suppose that $p,q$ are distinct prime
numbers, $P\subseteq T$ (resp. $Q\subseteq T$) is a Sylow
$p$-subgroup (resp. normal Sylow $q$-subgroup) and $T=PQ$. By
assumption, there is a subgroup $Q_0\subseteq Q$ satisfying
$X^{Q_0}\neq\emptyset$ and $[Q:Q_0]\leq C_0$. By Corollary
\ref{cor:Jordan-classic} there is an abelian subgroup
$Q_a\subseteq Q_0$ satisfying $[Q_0:Q_a]\leq\Jor_{2n}$ and
$Q_a$ can be generated by $n$ elements. Hence we have
\begin{equation}
\label{eq:index-Q-Q-a}
[Q:Q_a]\leq C_0\Jor_{2n}.
\end{equation}
Applying the same argument to $P$ we conclude the existence of
an abelian subgroup $P_a\subseteq P$ satisfying $[P:P_a]\leq
C_0\Jor_{2n}$ and such that $P_a$ can be generated by $n$
elements.

The bound (\ref{eq:index-Q-Q-a}) implies that for any $g\in P$ we have
$$[Q_a:Q_a\cap gQ_ag^{-1}]\leq C_0\Jor_{2n}.$$
Since $Q_a$ can be generated by $d$ elements, the group $Q_b=\bigcap_{g\in P}gQ_ag^{-1}$ satisfies
$$[Q_a:Q_b]\leq ((C_0\Jor_{2n})!)^d$$
(see e.g. \cite[Corollary 3.2]{M3}).
Since $Q_b\subset Q_0$ and $X^{Q_0}\neq\emptyset$, we have $X^{Q_b}\neq\emptyset$.
By definition $Q_b$ is a normal subgroup of $T$, so the action of $P_a$ on $X$ preserves
$X^{Q_b}$. By Lemma \ref{lemma:cohom-no-augmenta} the number of connected components of
$X^{Q_b}$ is $H^*(X)$-bounded, so there exists a subgroup $P_b\subseteq P_a$ such that
$[P_a:P_b]$ is $H^*(X)$-bounded and $P_b$ preserves each connected component of $X^{Q_b}$.

Let $Y\subseteq X^{Q_b}$ be any connected component, and let $N\to Y$ be the normal bundle of
its inclusion in $X$. The actions of $Q_b$ and $P_b$ on $Y$ lift to actions on $N$.
Consider the splitting
\begin{equation}
\label{eq:split-N-cx}
N\otimes\CC=\bigoplus_{\xi\in\wh{Q_b}}N^{Q_b,\xi}
\end{equation}
given by the action of $Q_b$ on $N$ (see Subsection
\ref{ss:decompose-equivariant-bundles}). Let $\Xi\subset
\wh{Q_b}$ be the set of characters $\xi$ such that
$N^{Q_b,\xi}\neq 0$. We have $|\Xi|\leq \dim X$, so $|\Xi|$ is
$H^*(X)$-bounded. The action by conjugation of $P_b$ on $Q_b$
induces an action on $\wh{Q_b}$ which necessarily preserves
$\Xi$, so there is a subgroup $P_c\subseteq P_b$ which fixes
each $\xi\in\Xi$ and such that $[P_b:P_c]$ is $H^*(X)$-bounded.
Since $P_c$ fixes each $\xi\in \Xi$, the action of $P_c$ on
$N\otimes\CC$ preserves the splitting (\ref{eq:split-N-cx}).
Since the action of $Q_b$ on each $N^{Q_b,\xi}$ is given by
scalar multiplication, it follows that the actions of $P_c$ and
$Q_b$ on $N\to Y$ commute. By the arguments in the proof of
Corollary \ref{cor:Jordan-classic} this implies that the
actions of $P_c$ and $Q_b$ on $X$ commute. Consequently,
$A:=P_cQ_b$ is an abelian group. Since $[T:A]\leq
[P:P_c][Q:Q_b]$ and both $[P:P_c]$ and $[Q:Q_b]$ are
$H^*(X)$-bounded, it follows that $[T:A]$ is $H^*(X)$-bounded.

Both $P_c$ and $Q_b$ can be generated by $n$ elements because
they are  subgroups of abelian subgroups that can be generated
by $n$ elements. Since the orders of $P_c$ and $Q_b$ are
coprime, this implies that $T$ can be generated by $n$
elements.
\end{proof}

\subsection{Proof of Theorems \ref{thm:main-ham} and \ref{thm:main-sympl}}
Theorem \ref{thm:main-ac-str} follows immediately from
combining Theorem \ref{thm:fixed-pt-symplectic} with Theorem
\ref{thm:test-Jordan} applied to the collection of all finite
subgroups of $\Symp(X,\omega)$.
Similarly, Theorem \ref{thm:main-ham} follows from Theorem
\ref{thm:fixed-pt-hamiltonian} and Theorem
\ref{thm:test-Jordan} applied to the collection of all finite
subgroups of $\Ham(X,\omega)$.

\subsection{Proof of Theorem \ref{thm:2-step-nilpotent}}
We begin with some preliminary results.

We claim that there is a constant $C_0$ such that for any finite subgroup
$\Gamma\subset\Symp(X,\omega)$ acting trivially on $H^1(X)$
and for any prime $p$, any $p$-subgroup $G\subseteq[\Gamma,\Gamma]$
has a subgroup $G_0\subseteq G$ such that $X^{G_0}\neq\emptyset$ and
$[G:G_0]\leq C$. Furthermore, $C_0$ only depends on $H^*(X)$.
This claim follows from Theorem \ref{thm:lift-commutator-action-line-bundle}
and Theorem \ref{thm:main-fixed-point}, using the same arguments as
in Subsection \ref{ss:proof-thm:fixed-pt-symplectic}.

Let $\cC$ be the collection of all finite subgroups $G\subset\Symp(X,\omega)$
for which there exist a finite subgroup $\Gamma\subset\Symp(X,\omega)$ acting
trivially on $H^1(X)$ and satisfying $G\subseteq[\Gamma,\Gamma]$. Combining
the previous claim with Theorem \ref{thm:test-Jordan} we conclude that
$\cC$ satisfies $\jJ(C_1,n)$ for some constant $C_1$ depending on $H^*(X)$.

To conclude the preliminaries, we recall that a well known lemma of Minkowski states
that for any natural number $k$ and any finite group $H\subseteq\GL(k,\ZZ)$ the restriction
to $H$ of the natural map $\GL(k,\ZZ)\to\GL(k,\ZZ/3)$ is injective (see
\cite{Mi}, and \cite[Lemma 1]{S1} for a modern exposition). Hence no finite
subgroup of $\Aut(H^1(X))$ has more than $3^{b_1(X)}$ elements.

We are now ready to prove Theorem \ref{thm:2-step-nilpotent}.

Let $\Gamma\subset\Symp(X,\omega)$ be a finite subgroup. Let $h^1:\Gamma\to\Aut(H^1(X))$
be the natural map. Its image $h^1(\Gamma)$ is a finite subgroup of $\Aut(H^1(X))$ so
by Minkowski's lemma
it has at most $3^{b_1(X)}$ elements. Hence $\Gamma_0=\Ker h^1\subset\Gamma$
satisfies $[\Gamma:\Gamma_0]\leq 3^{b_1(X)}.$
The commutator $G:=[\Gamma_0,\Gamma_0]$ is an element of $\cC$, so there is an
abelian subgroup $A\subseteq G$ satisfying $[G:A]\leq C_1$ and $A$ can
be generated by $n$ elements, where $\dim X=2n$.

For any prime $p$ let $A_p$ be the $p$-part of $A$. Since $A_p\subseteq[\Gamma_0,\Gamma_0]$,
there is a subgroup $A_{p,0}\subseteq A_p$ satisfying $X^{A_{p,0}}\neq\emptyset$ and
$[A_p:A_{p,0}]\leq C_0$. In particular, if $p>C_0$ then $A_{p,0}=A_p$. Hence setting
$A_0:=\prod_pA_{p,0}$ we have the following rough estimate
$$[A:A_0]\leq C_2:=C_0^{\pi(C_0)},$$
where $\pi(C_0)$ denotes the number of primes not bigger than $C_0$.

Define
$A_1=\bigcap_{\phi\in\Aut(G)}\phi(A_0)$.
We have $[G:A_1]\leq C_3$ for some $C_3$ depending only on $C_1C_2$ and $n$
(this is standard, see
e.g. \cite[Corollary 3.2]{M3}), and $A_1$ is a characteristic subgroup of $G$. Since $G$
is a normal subgroup of $\Gamma_0$, it follows that $A_1$ is normal in $\Gamma_0$.

Let $\Gamma_1=\Gamma_0/A_1$. Since $A_1$ is a normal subgroup of $G=[\Gamma_0,\Gamma_0]$, we have $[\Gamma_1,\Gamma_1]=[\Gamma_0,\Gamma_0]/A_1=G/A_1$ and $\Gamma_1/[\Gamma_1,\Gamma_1]=\Gamma_0/[\Gamma_0,\Gamma_0]$.
Hence the exact sequence $$1\to [\Gamma_1,\Gamma_1]\to \Gamma_1\to \Gamma_1/[\Gamma_1,\Gamma_1]\to 1$$ can be rewritten as
$$1\to G/A_1\to \Gamma_1\to \Gamma_0/[\Gamma_0,\Gamma_0]\to 1.$$
By Mann and Su's Theorem \ref{thm:MS}
there exists some constant $r$ depending on $H^*(X)$ such that
for any monomorphism $(\ZZ/p)^s\hookrightarrow \Gamma$ we have
$s\leq r$. This implies that the Sylow $p$-subgroups of $\Gamma$ can
be generated by at most $r(5r+1)/2$ elements (such bound
follows easily from Lemmas
\ref{lemma:maximal-normal-abelian} and \ref{lemma:GHMR}). Consequently, the
abelian group $\Gamma_0/[\Gamma_0,\Gamma_0]$ can be generated by $r(5r+1)/2$
elements.
By Lemma 2.2 in \cite{M1} there is an
abelian subgroup $$A_2\subseteq \Gamma_1$$ satisfying $[\Gamma_1:A_2]\leq
C_4$, where $C_4$ depends only on $|G/A_1|\leq C_3$ and $r(5r+1)/2$;
consequently, $C_4$ can be bounded above by a constant depending
only on $H^*(X)$.

Let $S$ be the preimage of $A_2$ under
the projection map $\Gamma_0\to \Gamma_1$.
Clearly $S$ is solvable, since it fits in an
exact sequence
$$0\to A_1\to S\to A_2\to 0,$$
and $A_1$, $A_2$ are abelian.
In general we cannot expect $S$ to be abelian or $2$-step nilpotent, since the action of $A_2$
by conjugation on $A_1$ might be nontrivial.
We want to prove that, nevertheless, $S$ contains an abelian or $2$-step nilpotent subgroup of bounded index.

For any prime $p$ we have $X^{A_{1,p}}\neq\emptyset$, where $A_{1,p}$ is the $p$-part
of $A_1\subseteq A_0$.
Let $N_p$ be the normal bundle of the inclusion $X^{A_{1,p}}\hookrightarrow X$
and let $N_p=\bigoplus_{\xi\in \wh{A_{1,p}}}N_{p,\xi}$ be its decomposition according
to the characters of $A_{1,p}$ (see Subsection \ref{ss:decompose-equivariant-bundles}).
Let $\Xi_p=\{\xi\in \wh{A_{1,p}}\mid N_{p,\xi}\neq 0\}$.
We may bound, as in Subsection \ref{ss:proof-of-thm-fixed-pt},
$$|\Xi_p|\leq 2n \sum_j b_j(X;\FF_p).$$
The right hand side can be bounded by a constant $B$ depending only on $H^*(X)$, not on $p$.
The action of $A_2$ on $A_1$ by conjugation preserves $A_{1,p}$ and hence permutes the elements of $\Xi_p$; this way we get a morphism
$\sigma_p:A_2\to\Perm(\Xi_p)$.
By Corollary \ref{cor:Jordan-classic} any $\alpha\in\Ker\sigma_p$ acts trivially on $A_{1,p}$
(this is the same idea that is used at the end of the proof of Theorem \ref{thm:test-Jordan}).
We have $[A_2:\Ker\sigma_p]\leq B!$, so the intersection
$$A_3:=\bigcap_{A'\subseteq A_2\atop [A_2:A']\leq B!}A',$$
acts trivially on $A_{1,p}$ for each prime $p$, and hence acts trivially on $A_1$.
Since the rank of $A_2$ is at most $r(5r+1)/2$,
we can bound $[A_2:A_3]\leq C_5$ where $C_5$ depends only on $r$ and $B$, hence only on $H^*(X)$.
Let $N\subseteq S$ be the preimage of $A_3$ under the projection map
$S\to A_2$. Then $N$ is abelian or $2$-step nilpotent and
\begin{align*}
[\Gamma:N] &= [\Gamma:\Gamma_0][\Gamma_0:S][S:N]=[\Gamma:\Gamma_0][\Gamma_1:A_2][A_2:A_3]\\
&\leq 3^{b_1(X)}C_4C_5=:C,
\end{align*}
where $C$ depends only on $H^*(X)$. So the proof of Theorem \ref{thm:2-step-nilpotent} is
complete.

\end{document}